\documentclass[11pt,reqno,a4paper]{amsart}

\usepackage{amsmath,amssymb,color}
\usepackage{amsfonts, amscd, epsfig, amsmath, amssymb,enumerate}
\usepackage{graphicx}
\usepackage{graphics}
\usepackage{color}
\usepackage{mathrsfs}
\usepackage[T1]{fontenc} 
\usepackage{mathptmx}
\usepackage{todonotes}
\usepackage{soul}
\usepackage{breqn}
\usepackage[margin=3cm]{geometry}

\newtheorem{theorem}{Theorem}[section]
\newtheorem{lemma}[theorem]{Lemma}
\newtheorem{proposition}[theorem]{Proposition}
\newtheorem{corollary}[theorem]{Corollary}
\newtheorem{remark}[theorem]{Remark}

\usepackage{tikz}
\usetikzlibrary{backgrounds}
\usetikzlibrary{patterns,fadings}
\usetikzlibrary{arrows,decorations.pathmorphing}
\usetikzlibrary{decorations}
\usetikzlibrary{calc}
\usetikzlibrary{shapes.misc}

\usepackage{setspace}
\usepackage[colorlinks=true,linkcolor=blue,citecolor=magenta]{hyperref}

\numberwithin{equation}{section}
\numberwithin{figure}{section}

\newcommand{\mc}[1]{{\mathcal #1}}

\newcommand{\<}{\big\langle}
\renewcommand{\>}{\big\rangle}

\renewcommand{\epsilon}{\varepsilon}

\newcommand{\R}{\mathbb R}
\newcommand{\Z}{\mathbb Z}
\newcommand{\N}{\mathbb N}
\renewcommand{\P}{\mathbb P}
\newcommand{\T}{\mathbb T}
\newcommand{\E}{\mathbb E}

\allowdisplaybreaks %%Pour \UTF{00E9}viter les grands espacements verticaux.

\title[Nonequilibrium fluctuations and moderate deviations]{Nonequilibrium fluctuations  and moderate deviations for the occupation time of the SSEP with Glauber dynamics}

\author{Linjie Zhao}
\address{School of Mathematics and Statistics, and Hubei Key Laboratory of Engineering Modeling and Scientific Computing, Huazhong University of Science and Technology, Wuhan 430074, China.}
\email{linjie\_zhao@hust.edu.cn}

\thanks{\textbf{Acknowledgments.}  The project is supported by the National Natural Science Foundation of China
	with grant numbers 12401168 and 12371142.}

\begin{document}

\maketitle

\begin{abstract}
We study the symmetric simple exclusion process with Glauber dynamics. When the process starts from a nonequilibrium measure, we prove central limit theorems for the occupation time in dimension two, and sample path moderate deviation principles in dimension one. For the fluctuations, we use the martingale method and the sharp relative entropy method from \cite{jaram18nonequilireaction}. For the moderate deviations, the main idea is to relate the occupation time to the density fluctuation field by using the logarithmic Sobolev inequality from the Glauber dynamics.
\end{abstract}

\section{Introduction}

The exclusion process is a prototype model in  interacting particle systems and statistical physics. In the dynamics, particles perform random walks on some graph subject to the exclusion rule, that is, there is at most one particle at each vertex of the underlying graph. On top of the exclusion process, we add a Glauber dynamics, which means that particles can be created or destroyed at each vertex according to some rate depending on the configuration. This model was studied by De Masi, Ferrari and Lebowitz \cite{de1986reaction} in order to derive reaction diffusion equations from microscopic models, and has received much attention since then.

In this paper, we study additive functionals of the process. The study of fluctuations of additive functionals of particle systems has a long history. For exclusion processes with stationary initial distributions, this issue was first studied in \cite{kipnis1987fluctuations} and \cite{kipnis1986central}. In the latter paper, Kipnis and Varadhan introduced the well-known martingale method. Since then, stationary fluctuations for the additive functionals have been investigated in various settings \cite{bernardin2004fluctuations,bernardin2016occupation,gonccalves2013scaling,li2008upper,seppalainen2003transience,sethuraman2000central,sethuraman2006centralcor,sethuraman2006superdiffusivity,sethuraman1996central}. We remark that in the previous literature  the process starts from stationary measures. Less is known when the initial measure is not stationary. Recently, Xu, Erhard and Franco \cite{erhard2024nonequilibrium} studied nonequilibrium fluctuations for the occupation time of the symmetric simple exclusion process (SSEP) in dimension one.  Their proof is based on self-duality of the SSEP and sharp  correlation estimates.  This result was extended by Xu and the author in \cite{xu2025nonequilibrium} to higher dimensions by using the martingale method and correlation estimates.  Both the previous papers used self-duality of the SSEP. For particle systems lack of self-duality, Xu and Fontes  \cite{fontes2021additive}   investigated  the weakly asymmetric simple exclusion process in dimension one by using  the sharp relative entropy bound  by Jara and Menezes \cite{jara2018non}.

The first main result of this paper concerns about nonequilibrium fluctuations of the occupation time.  Based on the techniques of \cite{fontes2021additive}, one can prove directly  the invariance principle for the degree-one additive functionals of the process in dimension one, see Propositions \ref{prop occup time invar} and \ref{prop add func invar}.  For this reason, we only state the result without proof and focus on higher dimensions.  In Theorem \ref{thm fluctuations in two dimension}, we prove the central limit theorem for the occupation time in dimension two in the sense of finite dimensional distributions. At this point, we are not aware of how to prove the tightness of the occupation time since it is hard to obtain sharp correlation estimates as for the SSEP in  \cite{xu2025nonequilibrium}.  It also remains open to extend the results to three or higher dimensions.

The proof of Theorem \ref{thm fluctuations in two dimension} follows from the martingale method. We decompose the occupation time as a martinale and some negligible term. Compared with the SSEP in  \cite{xu2025nonequilibrium}, we need to deal with higher-order terms due to the Glauber dynamics.  These terms can be killed by using the main lemma and sharp relative entropy bound from \cite{jaram18nonequilireaction}, which however only works in dimension two.

The second main result of this paper studies nonequilibrium moderate deviations of the occupation time. The scaling of moderate deviations is between large deviations and fluctuations. Large deviations for the occupation time of the exclusion process were studied in \cite{landim1992occupation,chang2004occupation}. Recently, Gao and Quastel in \cite{gao2025deviation} proved moderate deviation principles (MDP) for the occupation time of the SSEP on $\Z^d$ and in \cite{gao2024moderate} for the additive functionals of particle systems with mixing conditions at one  time point. Their results were extended to sample path MDP for the SSEP on regular trees by Xue \cite{xue2026equilibrium,xue2025central}, and for the SSEP  on $\Z$ by the author \cite{zhao2025moderate}. In Theorem \ref{thm occu mdp}, we prove sample path MDP for the occupation time of the process with non-stationary initial measures in dimension one, and extend the result to degree-one additive functionals in Corollary \ref{cor mdp degree one}.  As far as we know, this is the first result concerning about the MDP for the occupation time in the non-stationary setting. It remains an interesting question to extend the results to higher dimensions and to general additive functionals.

The main idea of the proof of Theorem \ref{thm occu mdp} and Corollary \ref{cor mdp degree one} is as follows.  In dimension one, one can relate the occupation time to the density fluctuation field of the process, as observed by Gon{\c c}alves and Jara \cite{gonccalves2013scaling} when considering equilibrium fluctuations of the occupation time. On the level of moderate deviations, this was done by the author in \cite{zhao2025moderate} in the stationary setting. When the initial measure is not stationary, we relate  the occupation time to the density fluctuation field in Lemma \ref{lem superexponential replacement} by using the logarithmic Sobolev inequality from the Glauber dynamics.  Once this is done, by using the MDP for the density fluctuation field proved in \cite{zhao2026moderate} and the contraction principle, we prove sample path MDP for the occupation time. 

The rest of the paper is as follows.  In section \ref{sec: model} we introduce the model and state the main results rigorously.  See Theorem \ref{thm fluctuations in two dimension} for the fluctuations of the occupation time in dimension two, and Theorem \ref{thm occu mdp} for the moderate deviations.  In section \ref{sec: preliminary} we introduce some preliminary results used throughout the proof. Sections \ref{sec: fluctuations} and \ref{sec:mdp} are devoted to the proof of fluctuations and moderate deviations respectively.

\section{Model and results}\label{sec: model}

The state space of the  model is $\Omega_{n}^d := \{0,1\}^{\T_n^d}$, where $\T_n^d := \Z^d / (n \Z^d)$ is the $d$-dimensional discrete torus with size $n \in \N:= \{1,2,\ldots\}$.    For a configuration $\eta \in \Omega_n^d$, $\eta_x \in \{0,1\}$ denotes the number of particles at site $x \in \T_n^d$. The infinitesimal generator of the process  is $L_n = n^2 L_n^{\rm ex} + L_n^{\rm r}$.  Here,  the generators $L_n^{\rm ex}$ and $L_n^{\rm r}$ correspond to the symmetric simple exclusion process (SSEP) and the Glauber dynamics respectively.  Precisely,  for any function $f: \Omega_n^d \rightarrow \R$,
\begin{align*}
	L_n^{\rm ex} f (\eta) &= \sum_{x,y \in \T_n^d \atop |x-y|=1}  [f(\eta^{x,y}) - f(\eta)],\\
	L_n^{\rm r} f (\eta) &= \sum_{x \in \T_n^d} c_x (\eta) [f(\eta^{x}) - f(\eta)].
\end{align*}
In this formula, the sum  $\sum_{x,y \in \T_n^d \atop |x-y|=1} $ is taken over all unordered bonds of $\T^d_n$. For $x,y \in \T_n^d$, $\eta^{x,y}$ is the configuration obtained from $\eta$ by swapping the values of $\eta_x$ and $\eta_y$,  and $\eta^x$ is the one by flipping the value of $\eta_x$, that is,
\[\eta^{x,y}_z = \begin{cases}
	\eta_x, \quad &\text{if } z=y,\\ 
	\eta_y, \quad &\text{if } z=x,\\
	\eta_z, \quad &\text{otherwise.}    
\end{cases} \qquad \eta^{x}_z = \begin{cases}
1- \eta_x, \quad &\text{if } z=x,\\ 
\eta_z, \quad &\text{otherwise.}    
\end{cases}\]
The flipping rate $c_x (\eta)$ is defined as
\[c_x (\eta) = \Big(a + \frac{\lambda}{2d} \sum_{y: |y-x|=1} \eta_y  \Big) (1-\eta_x) + b \eta_x,\]
where $a,b > 0$ and $\lambda > -a$ are given parameters.  Note that, when $a=0$, $L_n^{\rm r}$ is the generator of the standard contact process, see \cite{liggettips} for example.

For $\rho \in (0,1)$, let $\nu^n_\rho$ be the Bernoulli product measure on $\Omega_n^d$ with density $\rho$,
\[\nu^n_\rho (\eta_x = 1, \, \forall x \in A) = \rho^{|A|}, \quad \forall A \subset \T_n^d.\]
It is well known that the exclusion process with generator $L_n^{\rm ex}$ is reversible with respect to the measure $\nu^n_\rho$, see \cite{liggettips} for example. 

For $\rho \in [0,1]$, define
\[F(\rho) = E_{\nu^n_\rho} [c_x (\eta) (1-2\eta_x)] = (a+\lambda \rho) (1-\rho) - b \rho.\]
Since $F$ is a quadratic polynomial such that $F(0) = a > 0, F(1) = -b < 0$,  there exists a unique point $\rho_* \in (0,1)$ such that $F(\rho_*) = 0$.

Fix a time horizon $T > 0$. Let $\eta(t)$ be the process induced by the generator $L_n$ and initial measure $\nu_{\rho_*}^n$. Note that $\nu_{\rho_*}^n$ is not stationary for the process $\eta(t)$.  We omit the dependence of $\eta(t)$ on the scaling parameter $n$ to simplify notations. Let $\P^n_{\rho_*}$ be the measure on the path space $D ([0,T], \Omega_n^d)$, equipped with the Skorokhod topology, induced by the process $\eta (t)$. Denote by $\E^n_{\rho_*}$ the corresponding expectation. 

We are interested in the occupation time of the  process at the origin, which after rescaling is defined as
\[\Gamma^n (t) = \beta_{d,n} \int_0^t \bar{\eta}_0 (s)  ds,\]
where $\bar{\eta}_0 = \eta_0 - \rho_*$, and 
\begin{equation}
	\beta_{d,n} = \begin{cases}
		\sqrt{n}, \quad &\text{if} \quad d=1,\\
		\frac{n}{\sqrt{\log n}}, \quad &\text{if} \quad d=2,\\
		n, \quad &\text{if} \quad d\geq 3.
	\end{cases}
\end{equation}

\begin{remark}
	We rescale the occupation time by $\beta_{d,n}$ since this is the right scaling for the occupation time of the SSEP \cite{kipnis1987fluctuations}. Since the exclusion process is sped up by $n^2$ while the Glauber dynamics is not, it is natural to expect that this is also the right scaling for the SSEP with Glauber dynamics.
\end{remark}

\subsection{Fluctuations} In this subsection, we study fluctuations for the occupation time.   

\subsubsection{The case $d=1$}  To state the result in dimension one, we need to introduce the density fluctuation field of the process. Let $\mathcal{D} (\T)$ be the space of smooth functions on $\T$ equipped with the uniform topology, and  let $\mathcal{D}^\prime (\T)$ be its topological dual.  We equip  the space $D([0,T], \mathcal{D}^\prime (\T))$ of C{\`a}dl{\`a}g trajectories  with the uniform weak topology: a sequence $\{\mu^n_\cdot\}_{n \geq 1}$  converges to $\mu_\cdot$ in $D([0,T], \mathcal{D}^\prime (\T))$ if and only if for any $H \in \mathcal{D} (\T)$, 
\[\lim_{n \rightarrow \infty} \sup_{0 \leq t \leq T} |\<\mu^n_t,H\> - \<\mu_t,H\>| = 0.\]
The density fluctuation field of the process is defined as
\[\mathcal{Y}^n_t (H) = \frac{1}{\sqrt{n}} \sum_{x \in \T_n} \bar{\eta}_x (t) H(\tfrac{x}{n}), \quad H \in \mathcal{D} (\T).\]

The following result was  proved in \cite{jaram18nonequilireaction}, see also \cite[Proposition 4.1]{gonccalves2024clt}.

\begin{proposition}\label{prop density fluc}  There exists $\lambda_c > 0$ such that for any $\lambda \in [-\lambda_c,\lambda_c]$, the sequence of processes $\{\mathcal{Y}^n_t, 0 \leq t \leq T\}$ converges in distribution, as $n \rightarrow \infty$, in the space $D([0,T], \mathcal{D}^\prime (\T))$ to the process $\{\mathcal{Y}_t, 0 \leq t \leq T\}$, the unique solution to the following SPDE:
	\[\partial_t \mathcal{Y}_t =  \Delta \mathcal{Y}_t + F^\prime (\rho_*) \mathcal{Y}_t +  \sqrt{2 \chi (\rho_*)} \nabla \mathcal{\dot{W}}_t^1+ \sqrt{G(\rho_*)} \mathcal{\dot{W}}_t^2.\]
	Here, $\mathcal{\dot{W}}_t^1$ and $\mathcal{\dot{W}}_t^2$ are two independent $\mathcal{D}^\prime (\T)$-valued white noises, and \[\chi(\rho) = \rho (1-\rho), \quad G(\rho) = F(\rho) + 2b\rho.\] Moreover, the distribution of $\mathcal{Y}_0 (H)$ is  normal  with mean zero and variance $\chi (\rho_*) \|H\|_{L^2 (\T)}^2$ for any $H \in \mathcal{D} (\T)$. 
\end{proposition}

\begin{remark}
	Central limit theorems for the fluctuation field also hold for $d = 2,3$ in the sense of finite dimensional distributions. We state the result only in one dimension since this is what we need in this article.
\end{remark}

\begin{remark}
In the rest of this article, we fix $\lambda_c$ small enough so that  Proposition \ref{prop density fluc} and \eqref{lambda c condition} below hold, and choose the parameter $\lambda \in [-\lambda_c,\lambda_c]$.
\end{remark}

Let $C([0,T], \R)$ be the space  of continuous functions on $[0,T]$ equipped  with the uniform topology. Let $J \geq 0$ be a standard mollifier, that is,
\[\int_{\R} J (u) du= 1, \quad J (u) = J(-u), \quad {\rm supp}\; J \subset (-1,1).\]
For any $\varepsilon > 0$, define $J_{\varepsilon} (u) = \varepsilon^{-1} J(u/\varepsilon),\; u \in \R$.  Intuitively, for $\varepsilon$ small, 
\begin{align*}
	\int_0^t \mathcal{Y}^n_s (J_\varepsilon) ds &= \sqrt{n} \int_0^t \frac{1}{n} \sum_{x \in \T_n} \bar{\eta}_x (s) J_{\varepsilon} (\tfrac{x}{n}) ds\\
	&\approx \sqrt{n} \int_0^t \bar{\eta}_0 (s) ds.
\end{align*}
Thus, the limit of $\Gamma^n (t)$, as $n \rightarrow \infty$, should coincide with that of \[Z_t^\varepsilon := \int_0^t \mathcal{Y}_s (J_\varepsilon) ds\] as $\varepsilon \rightarrow 0$. This is the content of the following proposition. 

\begin{proposition}\label{prop occup time invar}
Let $\{\mathcal{Y}_t, 0 \leq t \leq T\}$ be the process introduced in Proposition \ref{prop density fluc}. The following statements hold.
\begin{itemize}
	\item[(i)] The process  $\{Z_t^\varepsilon, 0\leq t \leq T\}$ converges in distribution, as $\varepsilon \rightarrow 0$, to a Gaussian process $\{Z_t, 0 \leq t \leq T\}$ in the space $C([0,T], \R)$.
	\item[(ii)] The process $\{ \Gamma^n (t), 0 \leq t \leq T\}$ converges in distribution, as $n \rightarrow \infty$, to the same process $\{Z_t, 0 \leq t \leq T\}$ as in {\rm (i)} in the space $C([0,T], \R)$.
\end{itemize} 
\end{proposition}

The above result can be extended to additive functionals  with degree one. Precisely speaking, let $f: \Omega_{n}^d \rightarrow \R$ be a local function such that 
\begin{equation}\label{degree one}
\varphi_f (\rho_*) = 0, \quad \varphi_f^\prime (\rho_{*}) \neq 0,
\end{equation}
where $\varphi_f (\rho_{*}) = E_{\nu_{\rho_*}^n} [f]$.  Define
\[\Gamma^n_f (t) = \beta_{d,n} \int_0^t \big[f (\eta(s) )- \varphi_f (\rho_{*}) \big]ds.\]
Then, we have the following result.

\begin{proposition}\label{prop add func invar}
Assume that the local function $f: \Omega_{n}^1 \rightarrow \R$ satisfies \eqref{degree one}. Then, 
the process $\{ \Gamma_f^n (t), 0 \leq t \leq T\}$ converges in distribution, as $n \rightarrow \infty$, to the process $\{\varphi_f^\prime (\rho_{*}) Z_t, 0 \leq t \leq T\}$.
\end{proposition}

\begin{remark}
	The above results were proved by Fontes and Xu \cite{fontes2021additive} for the one dimensional weakly asymmetric simple exclusion process. The main idea is to relate the occupation time to the fluctuation field by proving the local Boltzmann-Gibbs principle in the nonequilibrium setting. Since their argument can be adapted to our case directly,  we omit it here.
\end{remark}

For later use, let us denote by $\alpha (\cdot,\cdot)$ the covariance function of the process $Z_t$,
\begin{equation}\label{alpha def}
\alpha (t,s) = \mathrm{Cov} (Z_t,Z_s), \quad s,t \geq 0.
\end{equation}

\subsubsection{The case $d=2$}  In dimension $d=2$,  the limit of the occupation time is given by the Brownian motion, which is the first main result of this article.

\begin{theorem}\label{thm fluctuations in two dimension}
The sequence  $\{ \Gamma^n (t), 0 \leq t \leq T\}$ converges in the sense of finite-dimensional distributions to the process $\{\sqrt{\chi (\rho_*) / \pi} B(t), 0 \leq t \leq T\}$, where $B(\cdot)$ is the standard one-dimensional Brownian motion.
\end{theorem}

\subsection{Moderate deviations}  In this subsection, we study moderate deviations for the occupation time in $d=1$. Let $\{a_n\}$ be a sequence of positive real numbers such that $1 \ll a_n \ll \sqrt{n}$, where $a_n \ll b_n$ denotes $\lim_{n \rightarrow \infty} a_n / b_n = 0$.
We will investigate  the rescaled occupation time $\Gamma^n (t) / a_n$. 

We first introduce the rate function. For any integer $k > 0$,  any sequence of times $0 \leq t_1 < t_2 < \ldots < t_k \leq T$,  and any path $\gamma \in C([0,T],\R)$,  let us denote $\Sigma_k = \Sigma_{t_1,\ldots,t_k}$ be the $k \times k$ matrix with the $(i,j)$-th entry given by $\alpha (t_i,t_j)$. Denote $\gamma_k = (\gamma (t_1), \ldots, \gamma (t_k))^{\rm T}$.  Then, the moderate deviation rate function is defined as
\[\mc{I}_{T} (\gamma) = \inf \Big\{  \frac{1}{2} \gamma_k^{\rm T} \Sigma_k^{-1} \gamma_k:  t_i \in [0,T] \; \text{for}\; 1 \leq i \leq k, \; k \geq 1\Big\}\]
if $\gamma (0)  = 0$, and define  $\mc{I}_{T} (\gamma) = + \infty$ otherwise.

Below is the second main result of this article.

\begin{theorem}\label{thm occu mdp}
	As $n \rightarrow \infty$, the sequence  $\{ \Gamma^n (t)/a_n, 0 \leq t \leq T\}$ satisfies the moderate deviation principles with decay rate $a_n^2$ and with rate function $\mc{I}_T$ in the space $C([0,T], \R)$. Precisely speaking,
	for any closed set $C$ and any open set $O$ in $C ([0,T], \R)$,
	\begin{align*}
		\limsup_{n \rightarrow \infty} \frac{1}{a_n^2} \log \P^n_{\rho_*} \Big( \{  \Gamma^n (t)/a_n , 0 \leq t \leq T\} \in C\Big) \leq - \inf_{\gamma \in C} \mathcal{I}_T (\gamma),\\
		\liminf_{n \rightarrow \infty} \frac{1}{a_n^2} \log \P^n_{\rho_*} \Big( \{ \Gamma^n (t)/a_n, 0 \leq t \leq T\} \in O\Big) \geq - \inf_{\gamma \in O} \mathcal{I}_T (\gamma).
	\end{align*}
\end{theorem}

The above result can also be extended to degree-one local functions.

\begin{corollary}\label{cor mdp degree one}
Assume that the local function $f: \Omega_{n}^1 \rightarrow \R$ satisfies \eqref{degree one}.  Then, the sequence  $\{ \Gamma^n_f (t)/a_n, 0 \leq t \leq T\}$ satisfies the moderate deviation principles with decay rate $a_n^2$ and with rate function $\mc{I}_T / \varphi_f^\prime (\rho_*)$ in the space $C([0,T], \R)$ as $n \rightarrow \infty$. 
\end{corollary} 

\section{Preliminaries}\label{sec: preliminary}

In this section, we present several previous results that will be used throughout the proof.

\subsection{Random walk}\label{subsec: rw} Let $p_n(t) := \{p_n(t,x), x \in \T_n^d\}$ be the transition probability of the accelerated continuous time simple random walk on $\T_n^d$, which jumps to one of its neighbors at rate $n^2$. Similarly, let $\mathfrak{p} (t) := \{\mathfrak{p} (t,x), x \in \Z^d\}$ be the transition probability of the continuous time simple random walk on $\Z^d$, which jumps to one of its neighbors at rate one. Then,
\begin{equation*}
	\partial_t p_n (t,x) = n^2 \Delta_{\rm D} p_n (t,x), \quad p_n (0,x) = \delta_0 (x).
\end{equation*} 
In this formula, $\Delta_{\rm D}$ is the discrete Laplacian, \[\Delta_{\rm D} p (x) = \sum_{i=1}^d \big(p(x+e_i) + p(x-e_i) - 2 p(x)\big),\] and $\delta_0$ is the Delta function at the origin. Define
\begin{equation}\label{gnx}
g_n(x) = \int_0^\infty e^{-t} p_n (t,x) dt, \quad x \in \T_n^d.
\end{equation}
Then,
\begin{equation}\label{gn p1}
	(1-n^2 \Delta_{\rm D}) g_n (x) = \delta_0 (x).
\end{equation}

The next lemma states some  estimates for the function $g_n$, see \cite{quastel2002central} for example.

\begin{lemma}\label{lem: rw}
\begin{itemize} 
	\item[(i)] There exists some constant $C$ independent of $n$ such that
	\[  \sum_{x \in \T_n^d} g_n(x)^2 \leq C \begin{cases}
		n^{-d}  \quad&\text{if}\quad d\leq 3;\\ 
		n^{-4} \log n \quad&\text{if}\quad d = 4;\\
		n^{-4}  \quad&\text{if}\quad d \geq  5.\end{cases}\] 
		
	\item[(ii)]  
	\[  \lim_{n \rightarrow +\infty} n^2 \beta_{d,n}^2 \sum_{x,y \in \T_n^d \atop |x-y|=1} \big(g_n (x) - g_n (y)\big)^2 = \begin{cases}
	(2\pi)^{-1}  \quad&\text{if}\quad d=2;\\ 
\int_0^\infty \mathfrak{p} (t,0) dt \quad&\text{if}\quad d\geq 3.\end{cases}\]
The same limit holds with the left hand side replaced by $\beta_{d,n}^2 g_n (0)$.
\end{itemize}
\end{lemma}

\subsection{Correlation estimate} The following two-point correlation estimate for the process was proved by De Masi, Ferrari and Lebowitz in \cite[Lemma 3.5]{de1986reaction}.

\begin{lemma}\label{lem: correlation}
There exists some constant $C$ independent of $n$ such that
\[\sup_{0 \leq s \leq T} \sup_{x,y \in \T_n^d \atop x \neq y}  \Big| \mathrm{Cov} (\eta_x (s),\eta_y (s))  \Big| \leq C n^{-d},\]
where the covariance  is with respect to the measure $\P^n_{\rho_*}$.
\end{lemma}

\subsection{Flow lemma} Let $q, q^\prime$ be two probability measures on $\T_n^d$. We say that $\phi : \T_n^d \times \T_n^d \rightarrow \R$ is a flow connecting $q$ and $q^\prime$ if it satisfies the following conditions: 
\begin{itemize}
	\item[(i)] $\phi  (x,y) = 0$ unless $|y-x|=1$;
	
	\item[(ii)] $\phi$ is anti-symmetric, \emph{i.e.}, $\phi  (x,y) = - \phi  (y,x)$ for any $x,y \in \T_n^d$;
	
	\item[(iii)] for any $x \in \T_n^d$,
	\[\sum_{y \in \T_n^d} \phi  (x,y) = q (x) - q^\prime (x).\]
\end{itemize}

For  any $x \in \T_n^d$ and any integer $\ell < n$, let $\Lambda_{x,\ell}$ be the cube of vertex $x$ with length $\ell$, 
\[\Lambda_{x,\ell} := \{y \in \T_n^d: 0 \leq y_i - x_i \leq \ell -1, \; \forall 1 \leq i \leq d\}.\]
Let $p_\ell$ be the uniform measure on $\Lambda_\ell:=\Lambda_{0,\ell}$, that is, $p_\ell (y) = \ell^{-d} \mathbf{1} \{y \in \Lambda_{\ell}\}$ for $y \in \T_n^d$. 
For $\ell < n/2$, define $q_\ell$ to be the convolution of $p_\ell$ with itself,
\[q_\ell (x) = p_\ell * p_\ell (x) = \sum_{y \in \T_n^d} p_\ell (y) p_\ell (x-y), \quad x \in \T_n^d.\]

The following result was proved by Jara and Menezes \cite{jara2018non}.

\begin{lemma}\label{lem flow}
There exists a flow $\phi_\ell$ connecting $\delta_0$ and $q_\ell$ such that
\begin{itemize}
	
	\item[(i)] $\phi_\ell (x,y) = 0$ if $x \notin \Lambda_{2\ell-1}$ or $y \notin \Lambda_{2\ell-1}$;

	\item[(ii)] there exists some constant $C_0 = C_0 (d)$ such that
	\[\sum_{x,y \in \T_n^d \atop |x-y|=1} \phi_\ell (x,y)^2 \leq C_0 g_d (\ell),\]
	where
	\begin{equation}
		g_d (\ell) = \begin{cases}
			\ell, \quad&\text{if}\quad d=1,\\
			\log	\ell, \quad&\text{if}\quad d=2,\\
			1, \quad&\text{if}\quad d \geq 3.
		\end{cases}
	\end{equation}
\end{itemize}
\end{lemma}

\subsection{Replacement estimate}\label{subsec replacement estimate}
Let $g: \T_n^d \rightarrow \R$ be a given function.  For any configuration $\eta \in \Omega_{n}^d$ and any $1 \leq i \leq d$, let us introduce
\[V_{i} (g,\eta) = \sum_{x \in \T_n^d} \bar{\eta}_x \bar{\eta}_{x+e_i} g(x).\]
For any $\ell < n/2$, define
\begin{equation}
\bar{\eta}_x^\ell = \sum_{y \in \T_n^d} q_\ell (y) \bar{\eta}_{x+y}, \quad V^\ell_{i} (g,\eta) = \sum_{x \in \T_n^d} \bar{\eta}_x \bar{\eta}^\ell_{x+e_i} g(x).
\end{equation}
By direct calculations, we can rewrite
\begin{equation}\label{V i ell}
	V^\ell_{i} (g,\eta) = \sum_{x \in \T_n^d} \overleftarrow{\eta}_x^{\ell} (g) \overrightarrow{\eta}^\ell_{x+e_i},
\end{equation}
where
\[\overleftarrow{\eta}_x^{\ell} (g) = \sum_{y \in \T_n^d} p_\ell (y) \bar{\eta}_{x-y} g(x-y), \quad \overrightarrow{\eta}^\ell_{x} = \sum_{y \in \T_n^d} p_\ell (y) \bar{\eta}_{x+y}.\]
 We also write $\overleftarrow{\eta}_x^{\ell} = \overleftarrow{\eta}_x^{\ell} (\mathbf{1})$, where $\mathbf{1}$ is the constant function equal to one. Let $\Gamma_n^{\rm ex}$ and $\Gamma_n^{\rm r}$ be the \emph{Carr{\' e} du champ} operators associated with the exclusion process and the Glauber dynamics respectively. Precisely, for functions $f,h: \Omega_{n}^d \rightarrow \R$,
 \begin{align*}
\Gamma_n^{\rm ex} (f,h) &= \frac{1}{2}\{L_n^{\rm ex} (fh) - f L_n^{\rm ex} h - h L_n^{\rm ex} f\}= \frac{1}{2}\sum_{x,y \in \T_n^d \atop |x-y|=1} [f(\eta^{x,y}) - f(\eta)] [h(\eta^{x,y}) - h(\eta)],\\
\Gamma_n^{\rm r} (f,h) &= \frac{1}{2}\{L_n^{\rm r} (fh) - f L_n^{\rm r} h - h L_n^{\rm r} f\} =\frac{1}{2} \sum_{x\in \T_n^d }c_x (\eta) [f(\eta^{x}) - f(\eta)] [h(\eta^{x}) - h(\eta)].
 \end{align*}
 We also write $\Gamma_n^{\rm ex} (f) = \Gamma_n^{\rm ex} (f,f)$ and similar notation for $\Gamma_n^{\rm r} $.

The following estimate when replacing $V_i$ by $V_i^\ell$ was proved in \cite[Equation (3.7)]{gonccalves2024clt}. 

\begin{lemma}\label{lem: replacement}
Let $f$ be a $\nu^n_{\rho_*}$-density. Then, for any $\gamma > 0$,
\begin{align*}
	\int (V_i (g,\eta) - V_i^\ell (g,\eta)) f (\eta)d \nu^n_{\rho_*} \leq \gamma \int \Gamma_n^{\rm ex} (\sqrt{f}) d \nu^n_{\rho_*} + \frac{1}{2 \gamma} \int \sum_{y,z \in \T_n^d \atop |y-z|=1} h_{y,z}^\ell (g, \eta)^2 f (\eta)d \nu^n_{\rho_*}.
\end{align*}
In this formula,
\[h_{y,z}^\ell (g, \eta) = \sum_{x \in \T_n^d} \phi_\ell (y-x,z-x) \bar{\eta}_{x-e_i} g(x-e_i)\]
with $\phi_\ell$ introduced in Lemma \ref{lem flow}.
\end{lemma}

\subsection{Relative entropy}  For any two probability measures $\mu, \nu$ on $\Omega_{n}^d$ such that $\mu \ll \nu$, recall that the relative entropy of $\mu$ with respect to $\nu$ is defined as 
\[H(\mu | \nu) = \int f \log f d \nu, \quad f = d \mu / d \nu.\]
Let $\mu^n_t$ be the distribution of the process at time $t$. The following relative entropy bound was first proved by Jara and Menezes \cite{jaram18nonequilireaction}, see also \cite{gonccalves2024clt}. 

\begin{lemma}\label{lem relative entropy}
There exists some constant $C$ such that for any $t > 0$,
\[H(\mu^n_t | \nu^n_{\rho_*}) \leq C n^{d-2} g_d (n).\]
\end{lemma}

\section{Fluctuations}\label{sec: fluctuations}

In this section, we prove Theorem \ref{thm fluctuations in two dimension}. For a configuration $\eta \in \Omega_{n}^d$, define 
\[G_n (\eta) = \sum_{x \in \T_n^d} g_n (x) \bar{\eta}_x,\]
where $g_n (x)$ was defined in \eqref{gnx}. By Dynkin's martingale formula,
\[M_n (t) = \beta_{d,n} G_n (\eta (t)) - \beta_{d,n} G_n (\eta (0)) - \beta_{d,n} \int_0^t L_n G_n (\eta(s)) ds\]
is a martingale.  By direct calculations, 
\[	n^2 L_n^{\rm ex} G_n (\eta) = \sum_{x \in \T_n^d} g_n (x) n^2 \Delta_{\rm D} \bar{\eta}_x 
= \sum_{x \in \T_n^d} \bar{\eta}_x n^2 \Delta_{\rm D}  g_n (x) 
= G_n (\eta) - \bar{\eta}_0.\]
In the second identity we used the summation by parts formula, and in the last one we used \eqref{gn p1}. For the Glauber dynamics, 
\[L_n^{\rm r} G_n (\eta) = \sum_{x \in \T_n^d} g_n (x) c_x (\eta) (1-2\eta_x).\]
Thus, 
\begin{equation}\label{martingale decomposition}
	\begin{aligned}
			\Gamma^n (t)= \beta_{d,n} G_n (\eta(0)) &- \beta_{d,n} G_n (\eta(t)) + \beta_{d,n} \int_0^t G_n (\eta(s)) ds \\
			&+ \beta_{d,n} \int_0^t \sum_{x \in \T_n^d} g_n (x) c_x (\eta (s)) (1-2\eta_x (s)) ds + M_n (t).
	\end{aligned}
\end{equation}

In the following five subsections, we deal with the five terms on the right hand side of the last equation respectively.

\subsection{The initial term} In this subsection, we prove that
\begin{equation}
	\lim_{n \rightarrow \infty} \E^n_{\rho_*} \big[\big\{ \beta_{d,n} G_n (\eta(0)) \big\}^2\big] = 0.
\end{equation}
Since the initial measure is a product measure, the last expectation equals
\[\chi (\rho_*) \beta_{d,n}^2 \sum_{x \in \T_n^d} g_n (x)^2,\]
which converges to zero for $d \geq 2$ by Lemma \ref{lem: rw}.

\subsection{The term $\beta_{d,n} G_n (\eta(t))$} We shall show that this term also vanishes in $L^2 (\P^n_{\rho_*})$. By developing the square, 
\begin{equation}\label{Gn square}
\begin{aligned}
	\E^n_{\rho_*} &\Big[\Big(\beta_{d,n} G_n (\eta(t))\Big)^2\Big] \\
	&= \beta_{d,n}^2 \sum_{x \in \T_n^d} g_n (x)^2 \E^n_{\rho_*} \big[\bar{\eta}_x (t)^2\big] + \beta_{d,n}^2 \sum_{x,y \in \T_n^d\atop x \neq y} g_n (x) g_n (y) \E^n_{\rho_*} \big[\bar{\eta}_x (t) \bar{\eta}_y (t)\big].
\end{aligned}
\end{equation}
Since $\eta_x$ is bounded by one, the first term on the right hand side is bounded by $\beta_{d,n}^2 \sum_{x \in \T_n^d} g_n (x)^2$, which converges to zero for $d \geq 2$ by Lemma \ref{lem: rw}.  For the second term, we write
\[\E^n_{\rho_*} \big[\bar{\eta}_x (t) \bar{\eta}_y (t)\big] = \mathrm{Cov} (\eta_x (t),\eta_y (t)) + (\E^n_{\rho_*} [\eta_x(t) ] - \rho_*)(\E^n_{\rho_*} [\eta_y(t) ] - \rho_*).\]
We claim that 
\begin{equation}\label{mean value}
|\E^n_{\rho_*} [\eta_x(t) ] - \rho_*| \leq C n^{-d}
\end{equation}
uniformly in $x$. Together with Lemma \ref{lem: correlation} and the fact that $\sum_{x} g_n (x) = 1$,  the second term on the right hand side of \eqref{Gn square} is bounded by $C n^{-d} \beta_{d,n}^2$, which also vanishes as $n \rightarrow \infty$ for $d \geq 2$. 

It remains to prove \eqref{mean value}. Denote
\[m_t = \E^n_{\rho_*} [\eta_x(t) ].\]
Note that the right hand side is independent of $x$ since both the initial measure and the dynamics are translation invariant.  Then,
\[m_t^\prime = \E^n_{\rho_*} [L_n \eta_x (t)] = \E^n_{\rho_*} [c_x (\eta(t)) (1-2\eta(t))].\]
By Lemma \ref{lem: correlation}, 
\[\E^n_{\rho_*} [c_x (\eta(t)) (1-2\eta(t))] = F(m_t) + \varepsilon_n (t),\]
where $|\varepsilon_n (t)| \leq C n^{-d}$ uniformly in $0 \leq t \leq T$. Thus,
\[m_t - \rho_*= \int_0^t F(m_s) ds + \int_0^t \varepsilon_n (s) ds.\]
Since $F(\rho_*) = 0$, by Taylor's expansion,
\[|m_t - \rho_*| \leq \|F^\prime\|_{L^\infty ([0,1])} \int_0^t |m_s - \rho_*| ds + C t n^{-d}. \]
We conclude the proof by Gr{\" o}nwall's inequality.

\begin{remark}\label{rmk betaG_nbounded}
By entropy inequality and Lemma \ref{lem relative entropy}, for any $\gamma > 0$,
\begin{align*}
\E^n_{\rho_*} [|\beta_{d,n} G_n (\eta (t))|] &\leq \frac{H(\mu^n_t | \nu_{\rho_*}^n)}{\gamma} + \frac{1}{\gamma} E_{\nu_{\rho_*}^n} [\exp \{ \gamma |\beta_{d,n} G_n (\eta)| \}]\\
&\leq C \Big( \frac{n^{d-2} g_d (n)}{\gamma} + \gamma \beta_{d,n}^2 \sum_{x \in \mathbb{Z}^d} g_n (x)^2 \Big).
\end{align*}
Optimizing over $\gamma$, we have
\[\E^n_{\rho_*} [|\beta_{d,n} G_n (\eta (t))|]  \leq C \beta_{d,n} \sqrt{n^{d-2} g_d (n) \sum_{x \in \mathbb{Z}^d} g_n (x)^2}.\]
Since the last term is bounded if $d \leq 3$, and diverges if $d \geq 4$, it is not sufficient for our purpose. 
\end{remark}

\begin{remark}
	We remark that this term is the only point where we need the correlation estimate from Lemma \ref{lem: correlation}.  It remains an interesting question to avoid using the correlation estimates so that the strategy presented in this article can be applied to more general models.
\end{remark}

\subsection{The time integral of the term $\beta_{d,n} G_n (\eta(s))$} It follows immediately from Cauchy-Schwarz inequality and the last subsection that 
\[\lim_{n \rightarrow \infty} \E^n_{\rho_*} \Big[\Big(\int_0^t \beta_{d,n} G_n (\eta(s)) ds \Big)^2\Big] \leq t^2 \lim_{n \rightarrow \infty} \sup_{0 \leq s \leq t}  \E^n_{\rho_*} \Big[\Big(\beta_{d,n} G_n (\eta(s))\Big)^2\Big]= 0.\]

\subsection{The fourth term on the right hand side of \eqref{martingale decomposition}} In this subsection, we deal with the term 
\[\beta_{d,n} \int_0^t \sum_{x \in \T_n^d} g_n (x) c_x (\eta (s)) (1-2\eta_x (s)) ds,\]
and shall show that it vanishes in $\P^n_{\rho_*}$-probability as $n \rightarrow 
\infty$.

Since $E_{\nu^n_{\rho_*}} [c_x (\eta) (1-2\eta_x)] = 0$, we have
\begin{equation}\label{fourth term}
c_x (\eta) (1-2\eta_x) = - (\lambda \rho_* + b) \bar{\eta}_x + \frac{\lambda(1-\rho_*)}{2d} \sum_{y: |y-x|=1} \bar{\eta}_y - \frac{\lambda}{2d} \sum_{y: |y-x|=1} \bar{\eta}_y \bar{\eta}_x.
\end{equation}
In the last subsection, we have shown that the first two terms on the right hand side of \eqref{fourth term}, after multiplying by $\beta_{d,n} g_n (x)$, then summing over $x$ and finally integrating the time from $0$ to $t$, converge to zero in $L^2 (\P^n_{\rho_*})$ as $n \rightarrow \infty$. 

The following result deals with the last term on the right hand side of \eqref{fourth term}.

\begin{lemma}\label{lem: second order term}
Assume that $d=2$. For any $0 \leq t \leq T$, any $1 \leq i \leq d$ and  any $\delta > 0$,
\[\lim_{n \rightarrow \infty} \P^n_{\rho_*} \Big( \Big|  \beta_{d,n} \int_0^t \sum_{x \in \T_n^d} g_n (x) \bar{\eta}_x (s) \bar{\eta}_{x+e_i} (s) ds\Big| > \delta\Big) = 0.\]
\end{lemma}

\begin{proof}
We only prove that
\begin{equation}\label{eqn 1}
	\lim_{n \rightarrow \infty} \P^n_{\rho_*} \Big(   \beta_{d,n} \int_0^t \sum_{x \in \T_n^d} g_n (x) \bar{\eta}_x (s) \bar{\eta}_{x+e_i} (s) ds > \delta\Big) = 0.
\end{equation}
By using  similar arguments, one can prove the other direction, thus
proving the lemma.

Fix some integer $\ell = \ell (n)> 0$ and constant $\theta = \theta (n)> 0$, which will be specified later. Define
\begin{equation}\label{B definition}
\begin{aligned}
	B_{\theta,\ell} (\eta) =&  \beta_{d,n} V_i^\ell (g_n,\eta) + \frac{2\theta  \beta_{d,n}^2}{n^2} \sum_{y,z \in \T_n^d \atop |y-z|=1} h_{y,z}^\ell (g_n, \eta)^2 \\
	&+ \frac{\lambda}{2 d \rho_* \theta } \sum_{i=1}^d V_i^\ell (\mathbf{1},\eta) + \frac{\lambda^2}{2d^2 \rho_*^2 n^2 \theta}  \sum_{y,z \in \T_n^d \atop |y-z|=1} h_{y,z}^\ell (\mathbf{1}, \eta)^2.
\end{aligned}
\end{equation}
Recall $V_i^\ell$ and $h_{y,z}^\ell $ from Subsection \ref{subsec replacement estimate}. The function $B_{\theta,\ell} (\eta)$ is carefully chosen so that
\begin{equation}\label{eqn 2}
\lim_{n \rightarrow \infty} \P_{\rho_*}^n \Big(   \int_0^t \big\{  \beta_{d,n} \sum_{x \in \T_n^d} g_n (x) \bar{\eta}_x (s) \bar{\eta}_{x+e_i} (s) - B_{\theta,\ell} (\eta(s)) \big\} ds > \delta/2\Big) = 0,
\end{equation}
\begin{equation}\label{eqn 3}
\lim_{n \rightarrow \infty} \P_{\rho_*}^n \Big(   \int_0^t B_{\theta,\ell} (\eta(s)) ds > \delta/2\Big) = 0.
\end{equation}
This is sufficient to prove \eqref{eqn 1}.

We first deal with \eqref{eqn 2}.  By Markov's exponential inequality and Feynman-Kac formula (see \cite[Lemma A.2]{jara2018non} for example),
\begin{align*}
\log 	&\P_{\rho_*}^n \Big(   \int_0^t \big\{  \beta_{d,n} \sum_{x \in \T_n^d}g_n (x) \bar{\eta}_x (s) \bar{\eta}_{x+e_i} (s) - B_{\theta,\ell} (\eta(s)) \big\} ds > \delta/2\Big) \\
&\leq - \delta \theta /2 + \log \E_{\rho_*}^n \Big[  \exp \Big\{ \int_0^t \theta\big\{  \beta_{d,n} \sum_{x \in \T_n^d} g_n (x) \bar{\eta}_x (s) \bar{\eta}_{x+e_i} (s) - B_{\theta,\ell} (\eta(s)) \big\} ds \Big\} \Big]\\
&\leq - \delta \theta /2 + t \alpha_n 
\end{align*}
where
\begin{align*}
\alpha_n = &\sup_{f: \nu^n_{\rho_{*}} {\rm density}} \Big\{ \int \theta\big\{  \beta_{d,n} \sum_{x \in \T_n^d} g_n (x) \bar{\eta}_x  \bar{\eta}_{x+e_i}  - B_{\theta,\ell} (\eta) \big\} f d \nu^n_{\rho_*} \\
&+ \int \frac{1}{2} L_n^* \mathbf{1} (\eta) f(\eta) d \nu^n_{\rho_*} - \int \Gamma_n (\sqrt{f}) d \nu^n_{\rho_*}  \Big\}.
\end{align*}
In this expression, $L_n^*$ is the adjoint of $L_n$ with respect to the measure $\nu_{\rho_*}^n$. Recall that
\[V_i (g_n,\eta)=\sum_{x \in \T_n^d} g_n (x) \bar{\eta}_x  \bar{\eta}_{x+e_i}.\]
Taking $\gamma = n^2/4$ in Lemma \ref{lem: replacement}, we have
\begin{align*}
\theta  \beta_{d,n}& \int (V_i (g_n,\eta) - V_i^\ell (g_n,\eta)) f (\eta)d \nu^n_{\rho_*} \\
&
\leq \frac{n^2}{4} \int \Gamma_n^{\rm ex} (\sqrt{f}) d \nu^n_{\rho_*} 
+ \frac{2 \theta^2  \beta_{d,n}^2}{ n^2} \int \sum_{y,z \in \T_n^d \atop |y-z|=1} h_{y,z}^\ell (g_n, \eta)^2 f (\eta)d \nu^n_{\rho_*}.
\end{align*}
By direct calculations (see Appendix \ref{subsec eqn13}),
\begin{equation}\label{eqn13}
\frac{1}{2} L_n^* \mathbf{1} = \frac{\lambda}{2d \rho_*} \sum_{x \in \T_n^d} \sum_{i=1}^d \bar{\eta}_x\bar{\eta}_{x+e_i} = \frac{\lambda}{2d \rho_*}  \sum_{i=1}^d V_i (\mathbf{1},\eta).
\end{equation}
Using Lemma \ref{lem: replacement} again, we have
\begin{equation}\label{ln 1 replacement}
\begin{aligned}
	\int & \frac{\lambda}{2d \rho_*}  \sum_{i=1}^d \big\{ V_i (\mathbf{1},\eta) - V_i^\ell (\mathbf{1},\eta) \big\} f d \nu^n_{\rho_*} \\
	&\leq  \frac{n^2}{4} \int \Gamma_n^{\rm ex} (\sqrt{f}) d \nu^n_{\rho_*}  + \frac{\lambda^2}{2d^2 \rho_*^2 n^2} \int \sum_{y,z \in \T_n^d \atop |y-z|=1} h_{y,z}^\ell (\mathbf{1}, \eta)^2 f (\eta)d \nu^n_{\rho_*}.
\end{aligned}
\end{equation}
By the definition of $B_{\theta,\ell}$ from \eqref{B definition}, we have $\alpha_n \leq 0$. Thus, the logarithm of the probability in \eqref{eqn 2} is bounded by $- \delta \theta / 2$. At the end of the proof, we shall take $\theta \gg 1$,  which proves $\eqref{eqn 2}$.

It remains to deal with \eqref{eqn 3}. By Markov's inequality, it suffices to show that
\begin{equation}\label{B vanish}
	\lim_{n \rightarrow \infty} E_{\mu^n_s} [|B_{\theta,\ell}|] = 0, 
\end{equation}
where $\mu^n_s$ is the distribution of the process at time $s$. We deal with the four terms in the definition of $B_{\theta,\ell}$ respectively. By entropy inequality (see \cite[page 338]{klscaling}) and Lemma \ref{lem relative entropy}, for any $\gamma > 0$,
\begin{align*}
	E_{\mu^n_s} \big[ \big| \beta_{d,n} V_i^\ell (g_n,\eta) \big| \big] &\leq \frac{H(\mu^n_s | \nu^n_{\rho_*})}{\gamma} + \frac{1}{\gamma} \log E_{\nu^n_{\rho_*}} \big[ \exp \big\{ \big| \gamma \beta_{d,n} V_i^\ell (g_n,\eta) \big|  \big\} \big] \\
	&\leq \frac{C n^{d-2} g_n (n)}{\gamma} + \frac{1}{\gamma} \log E_{\nu^n_{\rho_*}} \big[ \exp \big\{ \big| \gamma \beta_{d,n} V_i^\ell (g_n,\eta) \big|  \big\} \big].
\end{align*}
Since we will take $\gamma$ large enough, by using the basic inequalities $e^{|x|} \leq e^{x} + e^{-x}$ and $\log (a+b) \leq \log 2 + \max \{\log a, \log b\}$, we can remove the absolute value inside the above exponential. Recall from \eqref{V i ell} that
\[V_i^\ell (g_n, \eta) = \sum_{x \in \T_n^d} \overleftarrow{\eta}_x^\ell (g_n) \overrightarrow{\eta}_{x+e_i}^\ell.\]
Observe that $\overleftarrow{\eta}_x^\ell (g_n) \overrightarrow{\eta}_{x+e_i}^\ell$ and $\overleftarrow{\eta}_y^\ell (g_n) \overrightarrow{\eta}_{y+e_i}^\ell$ are independent under the measure $\nu^n_{\rho_*}$ if $|y-x| > 2 \ell$.  Together with H{\" o}lder's inequality,
\begin{equation}\label{eqn14}
\begin{aligned}
	\log &E_{\nu^n_{\rho_*}} \big[ \exp \big\{ \gamma \beta_{d,n} V_i^\ell (g_n,\eta)   \big\} \big] \\
	&\leq \frac{1}{(2\ell+1)^d} \sum_{x \in \T_n^d}  \log E_{\nu^n_{\rho_*}} \big[ \exp \big\{ \gamma \beta_{d,n} (2\ell+1)^d \overleftarrow{\eta}_x^\ell (g_n) \overrightarrow{\eta}_{x+e_i}^\ell  \big\} \big].
\end{aligned}
\end{equation}

To deal with the last logarithm, we first recall the notion of sub-Gaussian random variables. We say that a random variable $X$ is sub-Gaussian with variance $\sigma^2$ if for any $r \in \R$,
\[\log E [e^{r X}] \leq r^2 \sigma^2 / 2.\] 
Let $X$ and $Y$ be two sub-Gaussian random variables   with variances $\sigma_1^2$ and $\sigma_2^2$ respectively. By using Cauchy-Schwarz inequality, one can show that for any $0 < r < 1/(4 \sigma_1 \sigma_2)$,
\[E [e^{r X Y}] \leq 3,\]
see \cite{jara2018non} for example.

Now, we return to \eqref{eqn14}. One can check directly that $\overleftarrow{\eta}_x^\ell (g_n)$ and $\overrightarrow{\eta}_{x+e_i}^\ell$ are sub-Gaussian with variance of order $\ell^{-2d} \|g_n\|_{\ell^2 (\T_n^d)}^2$ and $\ell^{-d}$ respectively, where $\|g_n\|_{\ell^2 (\T_n^d)}^2 := \sum_{x \in \T_n^d} g_n (x)^2$. By taking $\gamma = \varepsilon_0 \ell^{d/2} \beta_{d,n}^{-1} \|g_n\|_{\ell^2 (\T_n^d)}^{-1}$ for some  constant $\varepsilon_0 >0$ small enough, 
\[E_{\nu^n_{\rho_*}} \big[ \exp \big\{ \gamma \beta_{d,n} (2\ell+1)^d \overleftarrow{\eta}_x^\ell (g_n) \overrightarrow{\eta}_{x+e_i}^\ell  \big\} \big] \leq 3.\]
Thus,
\[E_{\mu^n_s} \big[ \big| \beta_{d,n} V_i^\ell (g_n,\eta) \big| \big] \leq C \ell^{-d/2} \beta_{d,n} \|g_n\|_{\ell^2 (\T_n^d)} \big(n^{d-2}g_d (n) + n^d \ell^{-d}\big).\]
Similarly,
\begin{align}
E_{\mu^n_s} \Big[ \frac{\theta  \beta_{d,n}^2}{n^2} \sum_{y,z \in \T_n^d \atop |y-z|=1} h_{y,z}^\ell (g_n, \eta)^2 \Big] &\leq \frac{C \theta \ell^d g_d (\ell) \beta_{d,n}^2 \|g_n\|_{\ell^\infty (\T_n^d)}^2}{n^2} \big(n^{d-2}g_d (n) + n^d \ell^{-d}\big),\label{eqn15}\\
E_{\mu^n_s} \Big[ \Big|\frac{1}{\theta } \sum_{i=1}^d V_i^\ell (\mathbf{1},\eta)\Big|\Big] &\leq \frac{C}{\theta} \big(n^{d-2}g_d (n) + n^d \ell^{-d}\big),\label{V i ell bound}\\
E_{\mu^n_s} \Big[ \Big|\frac{1}{ n^2 \theta}  \sum_{y,z \in \T_n^d \atop |y-z|=1} h_{y,z}^\ell (\mathbf{1}, \eta)^2 \Big|\Big] &\leq \frac{C \ell^d g_d (\ell)}{n^2 \theta} \big(n^{d-2}g_d (n) + n^d \ell^{-d}\big).\label{h ell 1}
\end{align}
See Apprendix \ref{sec cal1} for detailed calculations of the last three estimates. Note that $\|g_n\|_{\ell^\infty (\T_n^d)} = g_n (0) \leq C \beta_{d,n}^{-2}$ by Lemma \ref{lem: rw}.  

In $d=2$, by taking $\ell = \lfloor n / \sqrt{\log n} \rfloor$, for $n$ large enough, we bound 
\[E_{\mu^n_s} [|B_{\theta,\ell}|] \leq C  \Big( n^{-1} \log n + \theta n^{-2} (\log n)^2 +  \theta^{-1}  \log n \Big).\]
We prove \eqref{B vanish} by taking  $\log n \ll \theta  \ll n^2 / (\log n)^{2}$.
\end{proof}

\begin{remark}
In dimensions $d \geq 3$, the best choice is to take $\ell = \lfloor n^{2/d} \rfloor$. Then, 
\[E_{\mu^n_s} [|B_{\theta,\ell}|] \leq C n^{d-2} \Big( \|g_n\|_{\ell^2 (\T_n^d)} + \theta n^{-2} + \theta^{-1}\Big).\]
The upper bound does not vanish as $n \rightarrow \infty$. This is the reason why our result is restricted to dimension $d=2$.
\end{remark}

\subsection{The martingale term} In this subsection, we prove the following invariance principle for the martingale term.

\begin{lemma}\label{lem martingale convergence}
For $d=2$, the sequence of martingales $\{M_n(t), 0 \leq t \leq T\}_{0 \leq t \leq T}$ converges in distribution, as $n \rightarrow \infty$, to $\{\sqrt{\chi (\rho_*) / \pi} B(t), 0 \leq t \leq T\}$ in the space $D([0,T],\R)$ endowed with the Skorokhod topology.
\end{lemma}

\begin{proof}
For any trajectory $m \in D([0,T],\R)$, define
\[J(m,T) = \sup \{ |m(t) - m(t-)|: 0 \leq t \leq T\},\]
where $m(t-) = \lim_{s \uparrow t} m(s)$. Let $\<M_n\>$	be the quadratic variation  of the martingale $M_n$. By \cite[Theorem 2.1]{whitt2007proofs}, it suffices to check the following conditions:
	\begin{itemize}
		\item[(i)] $\lim_{n \rightarrow \infty} \E^n_{\rho_*} [J(\<M_n\>,T)] = \lim_{n \rightarrow \infty} \E^n_{\rho_*} [J(M_n,T)^2] = 0$;
		\item[(ii)] for any $0 \leq t \leq T$,  $\langle M_n \rangle (t)$ converges in probability  to $t \chi (\rho_{*} )/ \pi$ as $n \rightarrow \infty$.  
	\end{itemize}

By direct calculations, the quadratic variation of the martingale equals
\begin{align*}
	\<M_n\> (t) &= \beta_{d,n}^2 \int_0^t \{ L_n G_n(\eta(s))^2 - 2 G_n (\eta(s)) L_n G_n (\eta(s)) \} ds\\
	&=\int_0^t n^2 \beta_{d,n}^2 \sum_{x,y \in \T_n^d \atop |x-y|=1} \big(g_n(x) - g_n(y)\big)^2 \big(\eta_x (s)- \eta_y (s) \big)^2 ds \\
	&\qquad + \int_0^t \beta_{d,n}^2 \sum_{x \in \T_n^d} c_x (\eta(s)) g_n (x)^2 ds.
\end{align*}
Since $\<M_n\>$ is continuous in  time $t$, $J(\<M_n\>, T) = 0$ for any $n$. From the definition of the martingale,
\[J(M_n,T) \leq \beta_{d,n} \sup_{x,y \in \T_n^d \atop |x - y|=1} |g_n (x) - g_n (y)| + \beta_{d,n} \|g_n\|_{\ell^\infty (\T_n^d)}.\]
By Lemma \ref{lem: rw}, the upper bound converges to zero as $n \rightarrow \infty$ for $d \geq 2$. This proves condition (i).

It remains to verify condition (ii).  Since $c_x (\eta)$ is uniformly bounded in $x$ and $\eta$, 
\[\int_0^t \beta_{d,n}^2 \sum_{x \in \T_n^d} c_x (\eta(s)) g_n (x)^2 ds \leq C t \beta_{d,n}^2 \sum_{x \in \T_n^d} g_n (x)^2 ,\]
which vanishes as $n \rightarrow \infty$ by Lemma \ref{lem: rw} in dimensions $d \geq 2$. Using Lemma \ref{lem: rw} again, it suffices to show that, for any $\delta > 0$,
\begin{equation}\label{eqn 4}
	\lim_{n \rightarrow \infty} \P_{\rho_*}^n \Big( \big|\int_0^t  n^2 \beta_{d,n}^2 \sum_{x,y \in \T_n^d \atop |x-y|=1} \big(g_n(x) - g_n(y)\big)^2 \big\{ \big(\eta_x (s)- \eta_y (s)  \big)^2 - 2 \chi (\rho_*) \big\}ds  \big| > \delta \Big) = 0.
\end{equation}
We first write
\[\big(\eta_x - \eta_y  \big)^2 - 2 \chi (\rho_*) = (1-2\rho_*) \bar{\eta}_x + (1-2\rho_*)  \bar{\eta}_y - 2 \bar{\eta}_x \bar{\eta}_y.\]
Then, we conclude the proof once we can show that, for any $1 \leq i \leq d$ and any $\delta > 0$, 
\begin{equation}\label{eqn 6}
	\lim_{n \rightarrow \infty} \P_{\rho_{*}}^n \Big( \big|\int_0^t  n^2 \beta_{d,n}^2 \sum_{x \in \T_n^d } \big(g_n(x) - g_n(x+e_i)\big)^2  \bar{\eta}_x (s) ds  \big| > \delta \Big) = 0,
\end{equation}
\begin{equation}\label{eqn 7}
	\lim_{n \rightarrow \infty} \P_{\rho_{*}}^n \Big( \big|\int_0^t  n^2 \beta_{d,n}^2 \sum_{x \in \T_n^d} \big(g_n(x) - g_n(x+e_i)\big)^2 \bar{\eta}_x (s) \bar{\eta}_{x+e_i} (s) ds  \big| > \delta \Big) = 0.
\end{equation}

For \eqref{eqn 6}, by Markov's inequality, Lemma \ref{lem: rw} and the fact that both the process and the initial measure are translation invariant, it suffices to show that
\begin{equation}\label{eqn 8}
\lim_{n \rightarrow \infty} \E_{\rho_*}^n \Big[\Big|\int_0^t \bar{\eta}_0 (s) ds \Big|\Big] = 0.
\end{equation}
From \eqref{martingale decomposition}, Remark \ref{rmk betaG_nbounded} and the above arguments, it is easy to see that
\[\limsup_{n \rightarrow \infty} \E_{\rho_*}^n \Big[\Big|\beta_{d,n} \int_0^t \bar{\eta}_0 (s) ds \Big|\Big] < + \infty,\]
which proves \eqref{eqn 8} since $\lim_{n \rightarrow \infty} \beta_{d,n} = +\infty$.

For \eqref{eqn 7}, we only prove that 
\begin{equation}\label{eqn 9}
	\lim_{n \rightarrow \infty} \P_{\rho_*}^n \Big( \int_0^t  n^2 \beta_{d,n}^2 \sum_{x \in \T_n^d} \big(g_n(x) - g_n(x+e_i)\big)^2 \bar{\eta}_x (s) \bar{\eta}_{x+e_i} (s) ds  > \delta \Big) = 0,
\end{equation}
and the other direction  can be proved in the same way. The proof is similar to that of Lemma \ref{lem: second order term}. For any integer $\ell > 0$, any $\theta > 0$ and any $x \in \T_n^d$, define
\begin{equation}\label{B frak}
\begin{aligned}
	\mathfrak{B}_{x,\ell, \theta} (\eta) =& \bar{\eta}_x \bar{\eta}_{x+e_i}^\ell + \frac{2c_n \theta}{n^2} \sum_{y,z \in \T_n^d \atop |y-z|=1} h_{y,z}^\ell (\delta_x,\eta)^2 \\
	&+ \frac{\lambda}{2d\rho_* c_n \theta} \sum_{j=1}^d V_j^\ell (\mathbf{1},\eta) +\frac{\lambda^2}{2d^2\rho_*^2 n^2 c_n \theta} \sum_{y,z \in \T_n^d \atop |y-z|=1} h_{y,z}^\ell (\mathbf{1},\eta)^2,
\end{aligned}
\end{equation}
where $c_n := n^2 \beta_{d,n}^2 \sum_{x \in \T_n^d} \big(g_n(x) - g_n(x+e_i)\big)^2$. Note that, by Lemma \ref{lem: rw}, $c_n$ is uniformly bounded in $n$. It suffices to show that
\begin{multline}\label{eqn 11}
	\lim_{n \rightarrow \infty} \P_{\rho_*}^n \Big( \int_0^t  n^2 \beta_{d,n}^2 \sum_{x \in \T_n^d} \big(g_n(x) - g_n(x+e_i)\big)^2 \\
	\times \{ \bar{\eta}_x (s) \bar{\eta}_{x+e_i} (s) - \mathfrak{B}_{x,\ell,\theta} (\eta(s)) \}ds  > \delta/2 \Big) = 0,
\end{multline}
\begin{equation}\label{eqn 12}
	\lim_{n \rightarrow \infty} \P_{\rho_*}^n \Big( \int_0^t  n^2 \beta_{d,n}^2 \sum_{x \in \T_n^d} \big(g_n(x) - g_n(x+e_i)\big)^2  \mathfrak{B}_{x,\ell,\theta} (\eta(s)) ds  > \delta/2 \Big) = 0.
\end{equation}

By Markov's exponential inequality, Jensen's inequality and Feynman-Kac formula, the logarithm of the probability in \eqref{eqn 11} is bounded by 
\begin{align*}
	- \delta \theta/2 
	+ \frac{n^2 \beta_{d,n}^2}{c_n}& \sum_{x \in \T_n^d} \big(g_n(x) - g_n(x+e_i)\big)^2  \\
	&\times \log \E^n_{\rho_*} \Big[ \exp \Big\{ c_n \theta \int_0^t [\bar{\eta}_x (s) \bar{\eta}_{x+e_i} (s) - \mathfrak{B}_{x,\ell,\theta} (\eta(s)) ] ds \Big\}\Big]\\
	\leq 	- \delta \theta/2  + \frac{n^2 \beta_{d,n}^2}{c_n} &\sum_{x \in \T_n^d} \big(g_n(x) - g_n(x+e_i)\big)^2  t \gamma_{n,x},
\end{align*}
where 
\begin{align*}
\gamma_{n,x} = \sup_{f: \nu_{\rho^n_*}{\rm-density}} \Big\{  &\int c_n \theta (\bar{\eta}_x \bar{\eta}_{x+e_i} - \mathfrak{B}_{x,\ell,\theta}  (\eta)) f (\eta) d \nu^n_{\rho_*} \\
&+ \int \frac{1}{2} L_n^* \mathbf{1}  (\eta)f (\eta) d \nu^n_{\rho_*} - \int \Gamma_n  (\sqrt{f}) d \nu^n_{\rho_*}\Big\}. 
\end{align*}
We will show that $\gamma_{n,x} \leq 0$, which allows us to bound the logarithm of the probability in \eqref{eqn 11} by $- \delta \theta/2$. Indeed, by choosing $g(y) = \delta_x (y)$ in Lemma \ref{lem: replacement}, 
\begin{align*}
	\int& c_n \theta (\bar{\eta}_x \bar{\eta}_{x+e_i}- \bar{\eta}_x \bar{\eta}_{x+e_i}^\ell ) f (\eta) d \nu^n_{\rho_*} \\
	&\leq \frac{n^2}{4} \int \Gamma_n^{\rm ex} (\sqrt{f}) d \nu^n_{\rho_*} 
	+ \int \frac{2c_n^2 \theta^2}{n^2} \sum_{y,z \in \T_n^d \atop |y-z|=1} h_{y,z}^\ell (\delta_x,\eta)^2 f (\eta) d \nu^n_{\rho_*}.
\end{align*}
Then, it is easy to see that $\gamma_{n,x} \leq 0$ by \eqref{ln 1 replacement} and the definition of $\mathfrak{B}_{x,\ell, \theta} (\eta)$.  

To prove \eqref{eqn 12}, by Markov's inequality, translation invariance of the process and the fact that $c_n$ is bounded, it suffices to show that
\begin{equation}
	\lim_{n \rightarrow \infty} E_{\mu^n_s} \big[ \big| \mathfrak{B}_{0,\ell,\theta}\big|\big] = 0.
\end{equation}
For the first term in the definition of $\mathfrak{B}_{0,\ell,\theta}$ from \eqref{B frak}, by entropy inequality and the fact that $\bar{\eta}_{e_i}^\ell$ is sub-Gaussian with variance of order $\ell^{-d}$ with respect to the measure $\nu^n_{\rho_*}$, for any $\gamma > 0$,
\begin{align*}
	E_{\mu^n_s} \big[ \big| \bar{\eta}_0 \bar{\eta}_{e_i}^\ell\big|\big] &\leq  	E_{\mu^n_s} \big[ \big| \bar{\eta}_{e_i}^\ell\big|\big]  \leq \frac{1}{\gamma} \Big(H (\mu^n_s | \nu^n_{\rho_*}) + \log E_{\nu^n_{\rho_*}} \big[ \exp \big\{ \gamma |\bar{\eta}_{e_i}^\ell| \big\} \big] \Big)\\
	&\leq C \Big( \frac{n^{d-2} g_d (n)}{\gamma} + \frac{\gamma}{\ell^d} \Big).
\end{align*}
For the second term in the definition of $\mathfrak{B}_{0,\ell,\theta}$, note that 
\[h_{y,z}^\ell (\delta_0,\eta) = \phi_\ell (y-e_i,z-e_i) \bar{\eta}_0.\]
Since $\bar{\eta}_0$ is bounded, by Lemma \ref{lem flow},
\[\frac{2c_n \theta}{n^2} \sum_{y,z \in \T_n^d \atop |y-z|=1} h_{y,z}^\ell (\delta_x,\eta)^2 \leq \frac{2 c_n \theta}{n^2} \sum_{y,z \in \T_n^d \atop |y-z|=1}  \phi_\ell (y-e_i,z-e_i)^2 \leq  \frac{C \theta}{n^2} g_d (\ell).\]
Recall that the last two terms in the definition of $\mathfrak{B}_{0,\ell,\theta}$ have been  dealt with in  Equations \eqref{V i ell bound} and \eqref{h ell 1}. Adding up the above sitimates, we have
\[E_{\mu^n_s} \big[ \big| \mathfrak{B}_{0,\ell,\theta}\big|\big] \leq C \Big(\frac{n^{d-2} g_d (n)}{\gamma} + \frac{\gamma}{\ell^d} + \frac{\theta g_d (\ell) }{n^2}+ \frac{1}{\theta} \big(1+\frac{\ell^dg_d(\ell)}{n^2}\big)\big(n^{d-2}g_d(n) + \frac{n^d}{\ell^d}\big)\Big).\]
In $d=2$, by choosing $\ell = \lfloor n / \sqrt{\log n}\rfloor$ and $\gamma = n$, we bound 
\[E_{\mu^n_s} \big[ \big| \mathfrak{B}_{0,\ell,\theta}\big|\big]  \lesssim \frac{\log n}{n} + \frac{\theta \log n}{n^2} + \frac{\log n}{\theta}.\]
Finally, we conclude the proof by taking $\log n \ll \theta \ll n^2 / \log n$.
\end{proof}

\begin{remark}
	The last lemma also holds for $d=3$ with the limiting variance $2 \chi( \rho_*) \int_0^\infty \mathfrak{p} (t,0) dt$. Indeed, in this case, we bound $E_{\mu^n_s} \big[ \big| \mathfrak{B}_{0,\ell,\theta}\big|\big]$ by a constant multiple of 
	\[\frac{n^{d-2}}{\gamma} + \frac{\gamma}{\ell^d} + \frac{\theta}{n^2} + \frac{1}{\theta} \Big(1+\frac{\ell^d}{n^2}\Big) \Big(n^{d-2} + \frac{n^d}{\ell^d}\Big). \]
	By taking $\ell = \lfloor n^{2/d}\rfloor, \gamma = n^{d/2}$, the last expression is bounded by $C (n^{d/2-2} + \theta / n^2 + n^{d-2} / \theta)$. For $d=3$, one can take $n \ll \theta \ll n^2$ to conclude the proof.
\end{remark}

\section{Moderate deviations}\label{sec:mdp}

In this section, we prove Theorem \ref{thm occu mdp} and Corollary \ref{cor mdp degree one}. In subsection \ref{sec: density fluc fields}, we recall the MDP for the density fluctuation field of the process proved in \cite{zhao2026moderate}. In subsection \ref{subsec exp tight}, we relate the occupation time to the density fluctuation field at the level of moderate deviations, from which the exponential tightness of the occupation time follows immediately.  Finite dimensional upper and lower bounds of the MDP are proved in subsections \ref{subsec upper mdp} and \ref{subsec lower mdp} respectively. We conclude the proof of Theorem \ref{thm occu mdp} in subsection \ref{subsec pf main thm}. Finally, we prove Corollary \ref{cor mdp degree one} in subsection \ref{subsec pf coro}.

\subsection{Density fluctuation fields}\label{sec: density fluc fields} For integer $n \geq 1$, let $\mu^n := \{\mu^n_t, 0 \leq t \leq T\}$ be the rescaled density fluctuation field,
$\mu^n_t = \mc{Y}^n_t / a_n.$
First, we need to introduce the moderate deviations rate function.  For $\mu_0 \in \mathcal{D}^\prime (\T)$, the rate function corresponding to the initial distribution is defined as
\begin{equation}\label{q 0}
	\mathcal{Q}_0 (\mu_0) := \sup_{\phi \in \mathcal{D} (\T)} \Big\{ \<\mu_0,\phi\> - \frac{\chi(\rho_*)}{2} \|\phi\|_{L^2 (\T)}^2 \Big\}.
\end{equation}
Let $\mc{C}^{1,\infty} ([0,T] \times \T)$ be the space of functions on $[0,T] \times \T$ which are continuously differentiable  in the time variable and smooth in the space variable. For $\mu \in D([0,T], \mathcal{D}^\prime (\T))$ and $H \in \mc{C}^{1,\infty} ([0,T] \times \T)$, define the linear functional
\[\ell_T (\mu,H) := \<\mu_T,H_T\> - \<\mu_0,H_0\> - \int_0^T \<\mu_s, (\partial_s+\Delta +F^\prime (\rho_*)) H_s\> ds.\]
The rate function corresponding to the dynamics of the process  is defined as
\[\mathcal{Q}_{\rm dyn} (\mu) := \sup_{H \in \mc{C}^{1,\infty} ([0,T] \times \T)} \Big\{ \ell_T (\mu,H)  -  \chi (\rho_*)  \|\nabla H\|_{L^2 ([0,T] \times \T)}^2 - \frac{G(\rho_*)}{2} \|H\|_{L^2 ([0,T] \times \T)}^2\Big\},\]
where $G(\rho)$ was defined in Proposition \ref{prop density fluc}. Finally, the rate function is defined as \[\mathcal{Q}_T = \mathcal{Q}_0 + \mathcal{Q}_{\rm dyn}.\]

The following nonequalibrium MDP for the process was proved in \cite{zhao2026moderate}.

\begin{proposition}\label{prop mdp fluc}
	The sequence  $\{\mu^n\}_{n \geq 1}$ satisfies the MDP with decay rate $a_n^2 $ and with rate function $\mathcal{Q}_T$. Precisely speaking, for any closed set $C$ and any open set $O$ in $D([0,T], \mathcal{D}^\prime (\T))$,
	\begin{align*}
		\limsup_{n \rightarrow \infty} \frac{1}{a_n^2} \log \P^n_{\rho_*} \big( \mu^n \in C\big) \leq - \inf_{\mu \in C} \mathcal{Q}_T (\mu),\\
		\liminf_{n \rightarrow \infty} \frac{1}{a_n^2} \log \P^n_{\rho_*} \big( \mu^n \in O\big) \geq - \inf_{\mu \in O} \mathcal{Q}_T (\mu).
	\end{align*}
\end{proposition}

Recall that $\{\mc{Y}_t, 0 \leq t \leq T\}$ is the limit of $\{\mc{Y}^n_t, 0 \leq t \leq T\}$ stated in Proposition \ref{prop density fluc}. Since the process $\{\mc{Y}_t, 0 \leq t \leq T\}$ is Gaussian, the sequence of processes $\{\mc{Y}_t / a_n, 0 \leq t \leq T\}$ satisfies the large deviation principles with decay rate $a_n^2$. Moreover, it turns out that the rate function is given by $\mc{Q}_T$ exactly the same as that of the sequence of processes $\{\mu^n_t, 0 \leq t \leq T\}$. Since the proof follows from the standard exponential martingale approach, we do not repeat the proof here, see \cite[Chapter 10]{klscaling} for example.  Actually, the proof is simpler in our case since we do not need the superexponential replacement lemma.

\begin{proposition}\label{prop mdp y_t}
	The sequence of processes $\{\mathcal{Y}_t / a_n, 0 \leq t \leq T\}_{n \geq 1}$ satisfies the large deviation principle with decay rate $a_n^2 $ and with rate function $\mathcal{Q}_T$. Precisely speaking, for any closed set $C$ and any open set $O$ in $D([0,T], \mathcal{D}^\prime (\T^d))$,
	\begin{align*}
		\limsup_{n \rightarrow \infty} \frac{1}{a_n^2} \log \P^n_{\rho_*} \Big( \{\mathcal{Y}_t/a_n, 0 \leq t \leq T\} \in C\Big) \leq - \inf_{\mu \in C} \mathcal{Q}_T (\mu),\\
		\liminf_{n \rightarrow \infty} \frac{1}{a_n^2} \log \P^n_{\rho_*} \Big( \{\mathcal{Y}_t/a_n, 0 \leq t \leq T\} \in O\Big) \geq - \inf_{\mu \in O} \mathcal{Q}_T (\mu).
	\end{align*}
\end{proposition}

%\begin{remark}
%	In the last proposition, we only need $a_n \gg 1$.
%\end{remark}

\subsection{Exponential tightness}\label{subsec exp tight}
To prove the exponential tightness of the sequence of processes $\{\Gamma^n (t) / a_n, 0 \leq t \leq T\}_{n \geq 1}$ in $C([0,T],\R)$, it suffices to verify the following two conditions:
\begin{enumerate}[(i)]
	\item 
	\begin{equation*}
		\lim_{M \rightarrow \infty} \limsup_{n \rightarrow \infty} \frac{1}{a_n^2} \log \P^n_{\rho_*} \Big( \sup_{0 \leq t \leq T}  \big| \Gamma^n (t) \big| > a_n M \Big) = - \infty;
	\end{equation*}
	\item for any $\kappa > 0$,
	\begin{equation*}
		\lim_{\delta \rightarrow 0} \limsup_{n \rightarrow \infty} \frac{1}{a_n^2} \log \P^n_{\rho_*} \Big( \sup_{|t -s |\leq \delta}  \big| \Gamma^n (t) - \Gamma^n (s)\big| > a_n \kappa \Big) = - \infty.
	\end{equation*}
\end{enumerate}

Recall the functions $J$ and $J_\varepsilon$ introduced before Proposition \ref{prop occup time invar}. In \cite[Lemma 3.7]{zhao2026moderate}, the author has proved that for any $\varepsilon > 0$, the process $\{\int_0^t \<\mu^n_t,J_\varepsilon\>, 0 \leq t \leq T\}$ satisfies the estimates in conditions (i) and (ii). Using Lemma \ref{lem superexponential replacement} below, it follows immediately that the process $\{ \Gamma^n (t)/ a_n, 0 \leq t \leq T\}$ also satisfies the above two conditions, thus concluding the proof of the exponential tightness.

\begin{lemma}\label{lem superexponential replacement}
	For any $\delta > 0$,
	\[\lim_{\varepsilon \rightarrow 0} \limsup_{n \rightarrow \infty} \frac{1}{a_n^2} \log \P^n_{\rho_*} \Big(\sup_{0 \leq t \leq T} \Big| \Gamma^n (t)/a_n - \int_0^t \<\mu^n_s,J_\varepsilon\> ds\Big| > \delta\Big) = - \infty.\]
\end{lemma}

\begin{proof}
	Observe that
	\[\frac{\sqrt{n}}{a_n} \bar{\eta}_0 (t) - \<\mu^n_t,J_\varepsilon\> = \frac{1}{a_n \sqrt{n}} \sum_{x \in \T_n} (\eta_0 (t) - \eta_x (t)) J_\varepsilon (x/n) + \frac{\sqrt{n}}{a_n} \mathcal{O} (n^{-1}),\]
	where the term $\mathcal{O} (n^{-1})$ comes from replacing $n^{-1} \sum_{x} J_\varepsilon (x/n)$ by one. Since $a_n \gg 1$, it suffices to prove that
	\begin{equation*}
		\lim_{\varepsilon \rightarrow 0} \limsup_{n \rightarrow \infty} \frac{1}{a_n^2} \log \P^n_{\rho_*} \Big(\sup_{0 \leq t \leq T} \Big| \int_0^t \frac{1}{a_n \sqrt{n}} \sum_{x \in \T_n} (\eta_0 (s) - \eta_x (s)) J_\varepsilon (x/n) ds \Big| > \delta\Big) = - \infty.
	\end{equation*}
	By Markov's inequality, we only need to prove that, for any $M > 0$,
	\begin{equation}\label{replacement 1}
		\lim_{\varepsilon \rightarrow 0} \limsup_{n \rightarrow \infty} \frac{1}{a_n^2} \log \E^n_{\rho_*} \Big[ \exp \Big\{ \sup_{0 \leq t \leq T} \Big| \int_0^t \frac{a_n M}{ \sqrt{n}} \sum_{x \in \T_n} (\eta_0 (s) - \eta_x (s)) J_\varepsilon (x/n) ds \Big| \Big\} \Big] = 0.
	\end{equation}
	
	Next, we want to remove the supermum over time inside the exponential. To this end, we recall the Garsia-Rodemich-Rumsey inequality, see \cite[page 182]{klscaling} for example.
	
	\begin{lemma}[Garsia-Rodemich-Rumsey inequality]\label{lem GRR inequality}
	Assume that $g: [0,T] \rightarrow \R$ is continuous and $g(0) = 0$, and that $\varphi (u), p(u)$ are strictly increasing functions such that
	\[\varphi (0) = p(0) = 0, \quad \lim_{u \rightarrow + \infty} \varphi (u) = + \infty.\]
	Then,
	\[\sup_{0 \leq t \leq T} |g(t)| \leq 8 \int_0^T \varphi^{-1} \big(4B/u^2\big) p(du),\]
	where
	\[B= \int_0^T ds \int_0^T dt \varphi \Big(\frac{|g(t)-g(s)|}{p(|t-s|)}\Big).\]
	\end{lemma}
	
	In the last lemma, we take 
	\begin{align*}
		g (t) = \int_0^t \frac{a_n M}{ \sqrt{n}} \sum_{x \in \T_n} (\eta_0 (s) - \eta_x (s)) J_\varepsilon (x/n) ds, \quad  \varphi (u) = u^{2q}, \quad p(u) = \sqrt{u},
	\end{align*}
	where $q > 2$ is fixed. Then, the supermum inside the exponential of \eqref{replacement 1} is bounded by
	\[C_1(q) \int_0^T (B/u^2)^{1/2q} u^{-1/2} du \leq C_2(q,T) B^{1/2q},  \]
	where 
	\[B = \int_0^T ds \int_0^T dt \Big| \frac{1}{\sqrt{|t-s|}} \int_s^t \frac{a_n M}{ \sqrt{n}} \sum_{x \in \T_n} (\eta_0 (\tau) - \eta_x (\tau)) J_\varepsilon (x/n) d\tau \Big|^{2q}.\]
	Introduce the function
	\[f_q (x) = \exp \Big\{  \Big((2q-1)^{2q} +x\Big)^{1/2q} \Big\}, \quad x \geq 0.\]
	One could check directly that $f_q$ is convex and that $f_q (x) \leq \exp \{2q-1 +x^{1/2q}\}$, see \cite[page 85]{zhao2025moderate} for example.  Thus,
	\begin{align*}
		\exp &\{C_2(q,T) B^{1/2q}\} \\
		&\leq f_q (C_2(q,T)^{2q} B) \\ &\leq \frac{1}{T^2} \int_0^T ds \int_0^T dt f_q \Big(C_2(q,T)^{2q} T^{2} \Big| \frac{1}{\sqrt{|t-s|}} \int_s^t \frac{a_n M}{ \sqrt{n}} \sum_{x \in \T_n} (\eta_0 (\tau) - \eta_x (\tau)) J_\varepsilon (x/n) d\tau \Big|^{2q} \Big)\\
		& \leq \frac{C_3(q)}{T^2} \int_0^T ds \int_0^Tdt \exp \Big\{ C_4(q,T) \Big| \frac{1}{\sqrt{|t-s|}} \int_s^t \frac{a_n M}{ \sqrt{n}} \sum_{x \in \T_n} (\eta_0 (\tau) - \eta_x (\tau)) J_\varepsilon (x/n) d\tau \Big| \Big\}.
	\end{align*}
	Thus, we only need to show that 
	\begin{multline}\label{eqn 5.3}
	\lim_{\varepsilon \rightarrow 0} \limsup_{n \rightarrow \infty} \frac{1}{a_n^2} \log \Big( \frac{C_3(q)}{T^2}  \int_0^T ds \int_0^Tdt\\  \E^n_{\rho_*} \Big[ \exp \Big\{  C_4(q,T) \Big| \frac{1}{\sqrt{|t-s|}} \int_s^t \frac{a_n M}{ \sqrt{n}} \sum_{x \in \T_n} (\eta_0 (\tau) - \eta_x (\tau)) J_\varepsilon (x/n) d\tau \Big| \Big\}\Big] \Big) =0.
	\end{multline}
	
	Using the basic inequality $e^{|x|} \leq e^{x} + e^{-x}$, we can  remove the absolute value inside the exponential. Then, by the nonequilibrium version of Feynman-Kac formula (see \cite[Proposition 4.3]{gonccalves2024clt} for example), we bound the logarithm of the last expectation by, for $s < t$,
	\[(t-s) \sup_{f: \nu^n_{\rho_*}{\rm-density}} \Big\{ \int \Big( \frac{C_4 a_n M}{ \sqrt{n} \sqrt{t-s}} \sum_{x \in \T_n} (\eta_0 - \eta_x) J_\varepsilon (x/n) + \frac{1}{2} L_n^* \mathbf{1} \Big) f d \nu^n_{\rho_*} - \int \Gamma_n (\sqrt{f}) d \nu_{\rho_*}^n  \Big\}. \]
	By first writing $\eta_0 - \eta_x = \sum_{y=0}^{x-1} (\eta_y - \eta_{y+1})$, then making the transformation $\eta \mapsto \eta^{y,y+1}$, and finally using Young's inequality, for any $\gamma > 0$, 
	\begin{align*}
		\int& \sum_{x \in \T_n} (\eta_0 - \eta_x) J_\varepsilon (x/n) f(\eta) d \nu_{\rho_*}^n \\
		&=\int \sum_{x \in \T_n} \sum_{y=0}^{x-1}(\eta_y - \eta_{y+1}) J_\varepsilon (x/n) f (\eta)d \nu_{\rho_*}^n \\
		&= \frac{1}{2} \int \sum_{x \in \T_n} \sum_{y=0}^{x-1}(\eta_y - \eta_{y+1}) J_\varepsilon (x/n) \big( f (\eta) - f(\eta^{y,y+1}) \big) d \nu_{\rho_*}^n \\
		&\leq \frac{\gamma}{8} \int \sum_{x \in \T_n} \sum_{y=0}^{x-1} (\eta_y - \eta_{y+1})^2 J_\varepsilon (x/n)^2 \big( \sqrt{f (\eta)} + \sqrt{f(\eta^{y,y+1})} \big)^2 d \nu_{\rho_*}^n \\
		&\qquad + \frac{1}{2 \gamma} \int \sum_{|x| < \varepsilon n} \sum_{y=0}^{x-1}  \big( \sqrt{f (\eta)} - \sqrt{f(\eta^{y,y+1})} \big)^2 d \nu_{\rho_*}^n\\
		&\leq C \gamma \varepsilon n^2 \|J_\varepsilon\|_{L^2 (\R)}^2 + \frac{\varepsilon n}{\gamma} \int \Gamma_n^{\rm ex} (\sqrt{f}) d \nu_{\rho_*}^n
	\end{align*}
	Taking $\gamma = 2 C_4 \varepsilon a_n M / (\sqrt{t-s}  n^{3/2})$, we bound 
	\[\int \frac{C_4 a_n M}{ \sqrt{n} \sqrt{t-s}} \sum_{x \in \T_n} (\eta_0 - \eta_x) J_\varepsilon (x/n)  f d \nu^n_{\rho_*} \leq \frac{C_5 (q,T) \varepsilon^2 a_n^2 M^2}{t-s} \|J_\varepsilon\|_{L^2 (\R)}^2+ \frac{n^2}{2} \int \Gamma_n^{\rm ex} (\sqrt{f}) d \nu_{\rho_*}^n.\]
	
Next, we bound the term $L_n^* \mathbf{1}$. From \eqref{eqn13} and \eqref{ln 1 replacement}, we have
\begin{align}
	\int \frac{1}{2} L_n^* \mathbf{1} (\eta) f(\eta) d \nu_{\rho_*}^n \leq& \frac{n^2}{4} \int \Gamma_n^{\rm ex} (\sqrt{f}) d \nu_{\rho_*}^n \notag\\
	&+ \frac{\lambda^2}{2d^2 \rho_*^2 n^2}\int \sum_{y,z \in \T_n^d \atop |y-z|=1} h^\ell_{y,z} (\mathbf{1},\eta)^2 f(\eta) d \nu_{\rho_*}^n\label{eqn5.4}\\
	&+ \frac{\lambda}{2d\rho_{*}} \int \sum_{i=1}^d V_i^\ell (\mathbf{1},\eta) f(\eta) d \nu_{\rho_*}^n.\label{eqn5.5}
\end{align}
Here, we present the calculations for dimension $d$ to explain why the methods only work in dimension one. By entropy inequality and H{\"o}lder's inequality, we bound \eqref{eqn5.4} by, for any $\gamma > 0$, 
\begin{multline*}
\frac{H(f | \nu_{\rho_*}^n)}{\gamma} + \frac{1}{\gamma} \log E_{\nu_{\rho_*}^n} \Big[\exp \Big\{  \frac{\gamma \lambda^2}{2d^2 \rho_*^2 n^2} \sum_{y,z \in \T_n^d \atop |y-z|=1} h^\ell_{y,z} (\mathbf{1},\eta)^2 \Big\}\Big]\\
\leq \frac{H(f | \nu_{\rho_*}^n)}{\gamma} + \frac{1}{\gamma (2\ell+1)^d} \sum_{y,z \in \T_n^d \atop |y-z|=1} \log E_{\nu_{\rho_*}^n} \Big[\exp \Big\{  \frac{\gamma \lambda^2 (2\ell+1)^d}{2d^2 \rho_*^2 n^2} h^\ell_{y,z} (\mathbf{1},\eta)^2 \Big\}\Big].
\end{multline*}
The Logarithmic Sobolev inequality for the Glauber dynamics states that there exists $\kappa (\rho_{*}) > 0$ such that
\begin{equation}\label{log sob ineq}
H(f | \nu_{\rho_*}^n) \leq \kappa (\rho_{*}) \int \Gamma_n^{\rm r} (\sqrt{f}) d \nu_{\rho_*}^n,
\end{equation}
see \cite[Lemma 3.1]{gonccalves2024clt} for example.  As in the proof of \eqref{h ell 1}, choosing $\gamma \lambda^2 (2\ell+1)^d g_d (\ell) = 2 \varepsilon_0 d^2 \rho_*^2 n^2$ for some small $\varepsilon_0 > 0$, we bound \eqref{eqn5.4} by
\[\frac{C_6 (d,\rho_*,\varepsilon_0) \lambda^2 \ell^d g_d (\ell)}{n^2} \Big( \kappa (\rho_{*}) \int \Gamma_n^{\rm r} (\sqrt{f}) d \nu_{\rho_*} + \frac{n^d}{\ell^d}\Big).\]
Similarly, we bound \eqref{eqn5.5} by, for any $\gamma > 0$,
\[\frac{1}{\gamma} \Big(\kappa (\rho_{*}) \int \Gamma_n^{\rm r} (\sqrt{f}) d \nu_{\rho_*} + \frac{1}{d(2\ell+1)^d} \sum_{i=1}^d \sum_{x \in \T_n^d} \log E_{\nu_{\rho_*}^n} \Big[\exp \Big\{  \frac{\gamma \lambda (2\ell+1)^d}{2 \rho_{*}} \overleftarrow{\eta}^\ell_x \overrightarrow{\eta}^\ell_{x+e_i} \Big\}\Big] \Big).\]
Choosing $\gamma \lambda = 2 \rho_* \varepsilon_0$ for some small $\varepsilon_0 > 0$, the last line is bounded by 
\[C_7 (d,\rho_*,\varepsilon_0)  \lambda \Big(\kappa (\rho_{*}) \int \Gamma_n^{\rm r} (\sqrt{f}) d \nu_{\rho_*} + \frac{n^d}{\ell^d} \Big).\]
In $d=1$, taking $\ell = \lfloor n/8 \rfloor$ and $\lambda_c$ small enough such that
\begin{equation}\label{lambda c condition}
C_6 \lambda_c^2 \kappa / 64 + C_7 \lambda_c \kappa < 1/2,
\end{equation}
we bound
\begin{equation}\label{Ln* 1}
\int \frac{1}{2} L_n^* \mathbf{1} (\eta) f(\eta) d \nu_{\rho_*}^n  \leq \frac{n^2}{4} \int \Gamma_n^{\rm ex} (\sqrt{f}) d \nu_{\rho_*}^n + \frac{1}{2} \int \Gamma_n^{\rm r} (\sqrt{f}) d \nu_{\rho_*}^n + C_8 (d,\rho_*,\varepsilon_0).
\end{equation}

	Adding up the above estimates, the expression in \eqref{eqn 5.3} is bound by \[C_9 (d,\rho_*,\varepsilon_0,T)  \big(M^2 \varepsilon^2 \|J_\varepsilon\|_{L^2 (\R)}^2 + a_n^{-2}\big).\] 
	This concludes the proof since $a_n \gg 1$ and $\varepsilon^2 \|J_\varepsilon\|_{L^2 (\R)}^2 \leq C \varepsilon$.
\end{proof}

\subsection{Finite dimensional upper bound}\label{subsec upper mdp}	In this subsection, we prove moderate deviation upper bound for the occupation time in the sense of finite dimensional distributions. For any integer $k > 0$, fix a sequence of times $t_1 < t_2 < \ldots < t_k$. 

\begin{lemma}\label{lem finite upper bound}
	Let the subsets $F_i \subset \R$, $1 \leq i \leq k$, be closed. Then,
	\[\limsup_{n \rightarrow \infty} \frac{1}{a_n^2} \log \P^n_{\rho_*} \Big( \frac{1}{a_n} \Gamma^n (t_i) \in F_i, 1 \leq i \leq k\Big) \leq - \inf \big\{  \frac{1}{2} \gamma^{\rm T} \Sigma^{-1} \gamma: \gamma_i \in F_i, 1 \leq i \leq k\big\}.\]
	In this formula, $\gamma = (\gamma_1, \gamma_2, \ldots, \gamma_k)^{\rm T}$, and $\Sigma = \Sigma_{t_1,\ldots,t_k}$ is the $k \times k$ matrix with the $(i,j)$-th entry equal to $\alpha (t_i,t_j)$ defined in \eqref{alpha def}. 
\end{lemma}

\begin{proof}
	For any $\delta > 0$ and any set $F \subset \R$, let $F^\delta$ be the $\delta$-dilation   of the set $F$, that is, $F^\delta = \{x: |x-y| \leq \delta\;\text{for some}\; y \in F\}$. By Lemma \ref{lem superexponential replacement}, Proposition \ref{prop mdp fluc} and the contraction principle, the $\limsup$ in the lemma is bounded by
	\begin{multline}\label{upper 1}
		\liminf_{\varepsilon \rightarrow 0} \;\limsup_{n \rightarrow \infty} \frac{1}{a_n^2} \log \P^n_{\rho_*} \Big( \int_0^{t_i} \<\mu^n_s,J_\varepsilon\> ds \in F_i^\delta, \;1 \leq i \leq k\Big)\\
		\leq - \sup_{\delta > 0}\; \limsup_{\varepsilon \rightarrow 0}\; \inf \Big\{ \mathcal{Q}_T (\mu):\;\int_0^{t_i} \<\mu_s,J_\varepsilon\> ds \in F_i^\delta, \;1 \leq i \leq k\Big\}.
	\end{multline}
	By Proposition \ref{prop mdp y_t} and the contraction principle, the sequence of vectors 
	$\big\{ Z^\varepsilon_{t_i}/a_n, \; 1 \leq i \leq k\big\}$
	satisfies the large deviation principle with decay rate $a_n^2$ and rate function 
	\[I^{(k)} (\gamma) := \inf \Big\{ \mathcal{Q}_T (\mu):\;\int_0^{t_i} \<\mu_s,J_\varepsilon\> ds = \gamma_i, \; 1 \leq i \leq k\Big\}\]
	for $\gamma \in \R^k$. Since the sequence of vectors $\big\{ Z^\varepsilon_{t_i}/a_n, \; 1 \leq i \leq k\big\}_{n \geq 1}$ is Gaussian, it also satisfies the large deviation principle with rate function $(1/2) \gamma^{\rm T} \Sigma_{\varepsilon}^{-1} \gamma$, where  $\Sigma_{\varepsilon}$ is the $k \times k$ matrix with the $(i,j)$-th entry equal to
	\[\alpha^\varepsilon (t_i,t_j) = {\rm Cov} (Z^{\varepsilon}_{t_i}, Z^\varepsilon_{t_j}).\]
	By uniqueness of the rate function, $I^{(k)} (\gamma) = (1/2) \gamma^{\rm T} \Sigma_{\varepsilon}^{-1} \gamma$.  By \eqref{upper 1}, we bound the $\limsup$ in  the lemma by
	\[- \sup_{\delta > 0} \limsup_{\varepsilon \rightarrow 0} \inf_{\gamma \in F_1^\delta \times \cdots \times F_k^\delta} \frac{1}{2}  \gamma^{\rm T} \Sigma_{\varepsilon}^{-1} \gamma.\]
	Since $F_1^\delta \times \cdots \times F_k^\delta$ is compact, for any $\varepsilon > 0$, let $\gamma_\varepsilon \in F_1^\delta \times \cdots \times F_k^\delta$ be such that
	\[\inf_{\gamma \in F_1^\delta \times \cdots \times F_k^\delta} \frac{1}{2}  \gamma^{\rm T} \Sigma_{\varepsilon}^{-1} \gamma = \frac{1}{2}  \gamma_{\varepsilon}^{\rm T} \Sigma_{\varepsilon}^{-1} \gamma_\varepsilon.\]
	Using the compactness of the set $F_1^\delta \times \cdots \times F_k^\delta$ again, there exists a subsequence $\{\varepsilon^\prime\}$ of $\{\varepsilon\}$ and a vector $\gamma_* \in F_1^\delta \times \cdots \times F_k^\delta$ such that $\lim_{\varepsilon^\prime \rightarrow 0} \gamma_{\varepsilon^\prime} = \gamma_*$.  Since $\lim_{\varepsilon \rightarrow 0} \alpha^\varepsilon (t_i,t_j)  = \alpha (t_i,t_j)$ by Proposition \ref{prop occup time invar}, we have that
	\begin{align*}
		\limsup_{\varepsilon \rightarrow 0} \inf_{\gamma \in F_1^\delta \times \cdots \times F_k^\delta} \frac{1}{2}  \gamma^{\rm T} \Sigma_{\varepsilon}^{-1} \gamma \geq \lim_{\varepsilon^\prime \rightarrow 0} \frac{1}{2}  \gamma_{\varepsilon^\prime}^{\rm T} \Sigma_{\varepsilon^\prime}^{-1} \gamma_{\varepsilon^\prime} = \frac{1}{2} \gamma_* \Sigma^{-1} \gamma_* \geq \inf_{\gamma \in F_1^\delta \times \cdots \times F_k^\delta} \frac{1}{2}  \gamma^{\rm T} \Sigma^{-1} \gamma.
	\end{align*}
	Thus, the $\limsup$ in the lemma is bounded from above by 
	\[- \sup_{\delta > 0} \inf_{\gamma \in F_1^\delta \times \cdots \times F_k^\delta} \frac{1}{2}  \gamma^{\rm T} \Sigma^{-1} \gamma = -\inf_{\gamma \in F_1 \times \cdots \times F_k} \frac{1}{2}  \gamma^{\rm T} \Sigma^{-1} \gamma,\]
	which concludes the proof.
\end{proof}

\subsection{Finite dimensional lower bound}\label{subsec lower mdp} In this subsection, we prove the following finite dimensional moderate deviation lower bound for the occupation time.

\begin{lemma}\label{lem finite lower bound}
	Let the subsets $O_i \subset \R$, $1 \leq i \leq k$, be open. Then,
	\[\liminf_{n \rightarrow \infty} \frac{1}{a_n^2} \log \P^n_{\rho_*} \Big( \Gamma^n (t_i) \in O_i, 1 \leq i \leq k\Big) \geq - \inf \big\{  \frac{1}{2} \gamma^{\rm T} \Sigma^{-1} \gamma: \gamma_i \in O_i, 1 \leq i \leq k\big\},\]
	where $\Sigma$ is the $k \times k$ matrix defined in Lemma \ref{lem finite upper bound}. 
\end{lemma}

\begin{proof}
	Choose any $\gamma \in O_1 \times \ldots \times O_k$. It suffices to show that 
	\begin{equation}\label{lower 1}
		\liminf_{n \rightarrow \infty} \frac{1}{a_n^2} \log \P^n_{\rho_*} \Big(  \Gamma^n (t_i) / a_n \in O_i, 1 \leq i \leq k\Big) \geq - \frac{1}{2} \gamma^{\rm T} \Sigma^{-1} \gamma.
	\end{equation}
	Choose $\delta > 0$ such that $(\gamma_i - 2 \delta, \gamma_i + 2 \delta) \subset O_i$ for every $1 \leq i \leq k$.  By Lemma \ref{lem superexponential replacement}, the left hand side of \eqref{lower 1} is bounded from below by 
	\begin{align*}
		\liminf_{n \rightarrow \infty}& \frac{1}{a_n^2} \log \P^n_{\rho_*} \Big( \Gamma^n (t_i) / a_n \in (\gamma_i - 2 \delta, \gamma_i + 2 \delta), 1 \leq i \leq k\Big)\\
		&\geq \limsup_{\varepsilon \rightarrow 0} \liminf_{n \rightarrow \infty} \frac{1}{a_n^2} \log \P^n_{\rho_*} \Big( \int_0^{t_i} \langle \mu^n_s, J_\varepsilon \rangle ds \in (\gamma_i -  \delta, \gamma_i +  \delta), 1 \leq i \leq k\Big).
	\end{align*}
	As in the proof of Lemma \ref{lem finite upper bound}, the last line is bounded from below by
	\begin{align*}
		- \liminf_{\varepsilon \rightarrow 0} \; \inf& \Big\{ \frac{1}{2}  \zeta^{\rm T} \Sigma_{\varepsilon}^{-1} \zeta: \zeta_i \in (\gamma_i - \delta, \gamma_i + \delta), \; 1 \leq i \leq k\Big\}\\
		&\geq - \liminf_{\varepsilon \rightarrow 0}  \frac{1}{2} \gamma^{\rm T} \Sigma_{\varepsilon}^{-1} \gamma =  -\frac{1}{2} \gamma^{\rm T} \Sigma^{-1} \gamma,
	\end{align*}
	thus concluding the proof.
\end{proof}

\subsection{Proof of Theorem \ref{thm occu mdp}}\label{subsec pf main thm} By Lemmas \ref{lem finite upper bound} and \ref{lem finite lower bound}, the finite dimensional distributions of $\{ \Gamma^n (t) / a_n, 0 \leq t \leq T\}$, at times $\{t_i\}_{1 \leq i \leq k}$, satisfy the moderate deviation principles with decay rate $a_n^2$ and with rate function $(1/2) \gamma^{\rm T} \Sigma_{t_1,\ldots,t_k}^{-1} \gamma$ with $\gamma = (\gamma_1,\ldots,\gamma_k)^{\rm T}$ and $\Sigma_{t_1,\ldots,t_k}$ defined in Lemma \ref{lem finite upper bound}.  Since we have proved in Subsection \ref{subsec exp tight} that the sequence of processes $\{ \Gamma^n (t)/ a_n, 0 \leq t \leq T\}$ is exponentially tight, we conclude the proof immediately by using \cite[Theorem 4.28]{feng2006large}. 

\subsection{Proof of Corollary \ref{cor mdp degree one}}\label{subsec pf coro}  In this subsection, we prove Corollary \ref{cor mdp degree one}. It suffices to show that, for any $\delta > 0$,
\[\limsup_{n \rightarrow \infty} \frac{1}{a_n^2} \log \P_{\rho_*}^n \Big(\sup_{0 \leq t \leq T} \Big| \Gamma_f^n(t) - \varphi_f^\prime (\rho_{*}) \Gamma_n (t)\Big| > a_n\delta\Big) = - \infty.\]
Note that for any local function $f: \Omega_{n}^1 \rightarrow \R$, $f - \varphi_f (\rho_{*}) - \varphi_f^\prime (\rho_{*})(\eta_0 - \rho_{*})$  can be expressed as a linear combination of $\prod_{x \in A} \bar{\eta}_x$, $A \subset \T_n$ finite, $|A| \geq 2$. Thus, it suffices to show that, for any finite $A \subset \T_n$ with $|A| \geq 2$ and for any $\delta > 0$,
\[\limsup_{n \rightarrow \infty} \frac{1}{a_n^2} \log \P_{\rho_*}^n \Big(\sup_{0 \leq t \leq T} \Big| \sqrt{n}  \int_0^t \prod_{x \in A} \bar{\eta}_x (s) ds \Big| > a_n\delta\Big) = - \infty.\]
Assume that $|A| = m \geq 2$ and label the points in $A$ by $x_1 < x_2 < \ldots < x_m$. Similar to the proof of Lemma \ref{lem superexponential replacement}, we have
\[\limsup_{\varepsilon \rightarrow 0} \limsup_{n \rightarrow \infty} \frac{1}{a_n^2} \log \P_{\rho_*}^n \Big(\sup_{0 \leq t \leq T} \Big| \sqrt{n}  \int_0^t \prod_{x \in A} \bar{\eta}_x (s)  - \overleftarrow{\eta}_{x_1}^{\varepsilon n} (s) [\prod_{x \in A \atop x \neq x_1,x_m} \bar{\eta}_x (s)]\overrightarrow{\eta}_{x_m}^{\varepsilon n} (s)  ds \Big| > a_n\delta\Big) = - \infty.\]
Then, it suffices to show that, for any $\varepsilon > 0$,
\[\limsup_{n \rightarrow \infty} \frac{1}{a_n^2} \log \P_{\rho_*}^n \Big(\sup_{0 \leq t \leq T} \Big| \sqrt{n}  \int_0^t  \overleftarrow{\eta}_{x_1}^{\varepsilon n} (s) [\prod_{x \in A \atop x \neq x_1,x_m} \bar{\eta}_x (s)]\overrightarrow{\eta}_{x_m}^{\varepsilon n} (s)  ds \Big| > a_n\delta\Big) = - \infty.\]
By Markov's inequality, we only need to show that for any $M > 0$, 
\[\limsup_{n \rightarrow \infty} \frac{1}{a_n^2} \log \E_{\rho_*}^n \Big[\exp \Big\{ \sup_{0 \leq t \leq T} \Big| a_n M \sqrt{n}  \int_0^t  \overleftarrow{\eta}_{x_1}^{\varepsilon n} (s) [\prod_{x \in A \atop x \neq x_1,x_m} \bar{\eta}_x (s)]\overrightarrow{\eta}_{x_m}^{\varepsilon n} (s)  ds \Big| \Big\} \Big] = 0.\]
By using Garsia-Rodemich-Rumsey inequality as in the proof of Lemma \ref{lem superexponential replacement}, we can  first remove the supremum over time inside the exponential. Then, by Feynman-Kac formula, the last expression on the left hand side is bounded by
\begin{equation}\label{feynman kac}
\begin{aligned}
	\frac{T}{a_n^2} \sup_{f: \nu_{\rho_*}^n {\rm density}} \Big\{& \int a_n M \sqrt{n} \overleftarrow{\eta}_{x_1}^{\varepsilon n}  [\prod_{x \in A \atop x \neq x_1,x_m} \bar{\eta}_x ]\overrightarrow{\eta}_{x_m}^{\varepsilon n}   f (\eta)d \nu_{\rho_*}^n \\&+ \int \frac{1}{2} L_n^* \mathbf{1} (\eta) f(\eta) d \nu_{\rho_*}^n 
	- \int \Gamma_n (\sqrt{f}) d\nu_{\rho_*}^n \Big\}.
\end{aligned}
\end{equation}
By entropy inequality and Logarithmic Sobolev inequality for the Glauber dynamics (see \eqref{log sob ineq}), we bound the first term inside the supremum by, for any $\gamma > 0$,
\[\frac{1}{\gamma} \Big\{ \kappa \int \Gamma_n^{\rm r} (\sqrt{f}) d \nu_{\rho_*}^n + \log E_{\nu_{\rho_*}^n} \big[ \exp \big\{ \gamma a_n M \sqrt{n} \overleftarrow{\eta}_{x_1}^{\varepsilon n}  [\prod_{x \in A \atop x \neq x_1,x_m} \bar{\eta}_x ]\overrightarrow{\eta}_{x_m}^{\varepsilon n}  \big\}\big] \Big\}.\]
Since $\overleftarrow{\eta}_{x_1}^{\varepsilon n}$ and $\overrightarrow{\eta}_{x_m}^{\varepsilon n}$ are both sub-Gaussian with variance of order $(\varepsilon n)^{-1}$ with respect to the measure $\nu_{\rho_*}^n$ and  $\prod_{x \in A \atop x \neq x_1,x_m} \bar{\eta}_x$ is bounded, by taking $\gamma a_n M\sqrt{n} = \varepsilon_0 \varepsilon n$ for small enough $\varepsilon_0 > 0$,  the last expression is bounded by
\[\frac{a_n M }{\varepsilon_0 \varepsilon \sqrt{n}} \big\{ \kappa \int \Gamma_n^{\rm r} (\sqrt{f}) d \nu_{\rho_*}^n + \log 3\big\}.\]
Here, we also used properties of sub-Gaussian random variables stated after \eqref{eqn14}. Since $a_n \ll \sqrt{n}$,   together with \eqref{Ln* 1}, for $n$ large enough, we bound \eqref{feynman kac} by $C/a_n^2$. This concludes the proof since $a_n \gg 1$.

\appendix

\section{Calculations}

\subsection{Proof of \eqref{eqn13}}\label{subsec eqn13} By direct calculations,
\[\begin{aligned} L_{n}^{*} \mathbf{1} &=  L_{n}^{{\rm r},*} \mathbf{1} = \sum_{x\in \T_n^d} \Big\{  \frac{\nu_{\rho_*}^n (\eta^x)}{\nu_{\rho_*}^n (\eta)} c_x (\eta^x) - c_x (\eta) \Big\}\\
	 & =\sum_{x \in \mathbb{T}_{n}^{d}}\left\{\eta_{x}\left(\frac{1-\rho_{*}}{\rho_{*}}\left[a+\frac{\lambda}{2 d} \sum_{y:|y-x|=1} \eta_{y}\right]-b\right)+\left(1-\eta_{x}\right)\left(\frac{\rho_{*}}{1-\rho_{*}} b-\left[a+\frac{\lambda}{2 d} \sum_{y:|y-x|=1} \eta_{y}\right]\right)\right\} \\ & =\sum_{x \in \mathbb{T}_{n}^{d}}\left(\frac{\eta_{x}}{\rho_{*}}-\frac{1-\eta_{x}}{1-\rho_{*}}\right)\left(\left[a+\frac{\lambda}{2 d} \sum_{y:|y-x|=1} \eta_{y}\right]\left(1-\rho_{*}\right)-b \rho_{*}\right) \\ & =\frac{1}{\chi\left(\rho_{*}\right)} \sum_{x \in \mathbb{T}_{n}^{d}} \bar{\eta}_{x}\left(F\left(\rho_{*}\right)+\frac{\lambda\left(1-\rho_{*}\right)}{2 d} \sum_{y:|y-x|=1} \bar{\eta}_{y}\right) \\ & =\frac{\lambda}{2 d \rho_{*}} \sum_{x \in \mathbb{T}_{n}^{d}} \sum_{y:|y-x|=1} \bar{\eta}_{x} \bar{\eta}_{y}.\end{aligned}\]
	 In the first identity, we used the fact that $L_n^{\rm ex,*} \mathbf{1} = 0$; in the last one, we used $F(\rho_*) = 0$.
 
 \subsection{Proof of \eqref{eqn15}-\eqref{h ell 1}}\label{sec cal1}
 
 We first prove \eqref{eqn15}. For any $\gamma > 0$,
 \begin{align*}
 	&E_{\mu^n_s} \Big[ \frac{\theta  \beta_{d,n}^2}{n^2} \sum_{y,z \in \T_n^d \atop |y-z|=1} h_{y,z}^\ell (g_n, \eta)^2 \Big] \\
 	&\leq \frac{H(\mu^n_s | \nu^n_{\rho_*})}{\gamma} + \frac{1}{\gamma} \log E_{\nu^n_{\rho_*}} \Big[ \exp \Big\{  \frac{\gamma \theta  \beta_{d,n}^2}{n^2} \sum_{y,z \in \T_n^d \atop |y-z|=1} h_{y,z}^\ell (g_n, \eta)^2 \Big\}\Big] \\
 	&\leq \frac{H(\mu^n_s | \nu^n_{\rho_*})}{\gamma} + \frac{1}{\gamma (2\ell+1)^d} \sum_{y,z \in \T_n^d \atop |y-z|=1} \log E_{\nu^n_{\rho_*}} \Big[ \exp \Big\{  \frac{\gamma (2\ell+1)^d\theta  \beta_{d,n}^2}{n^2}  h_{y,z}^\ell (g_n, \eta)^2 \Big\}\Big].
 \end{align*}
 Since $h_{y,z}^\ell (g_n, \eta)$ is sub-Gaussian with variance \[\sum_{x\in \T_n^d} \phi_\ell (y-x,z-x)^2 g_n (x-e_i)^2\leq  C_0 g_d (\ell) \|g_n\|_{\ell^\infty (\T_n^d)}^2,\] we conclude the proof by choosing $\gamma = \varepsilon_0 n^2 / (\theta \ell^d g_d (\ell) \beta_{d,n}^2 \|g_n\|_{\ell^\infty (\T_n^d)}^2)$ for $\varepsilon_0 > 0$ small enough.

Now, we prove \eqref{V i ell}. For $1 \leq i \leq d$, for any $\gamma > 0$, we bound
\begin{align*}
	E_{\mu^n_s} \Big[ \Big|\frac{1}{\theta } V_i^\ell (\mathbf{1},\eta)\Big|\Big] 
	&\leq \frac{H(\mu^n_s | \nu^n_{\rho_*})}{\gamma} + \frac{1}{\gamma} \log E_{\nu^n_{\rho_*}} \Big[ \exp \Big\{  |\frac{\gamma}{\theta } V_i^\ell (\mathbf{1},\eta)| \Big\}\Big] \\
	&\leq \frac{H(\mu^n_s | \nu^n_{\rho_*})}{\gamma} + \frac{1}{\gamma(2\ell+1)^d} \sum_{x \in \T_n^d} \log E_{\nu^n_{\rho_*}} \Big[ \exp \Big\{  \frac{\gamma(2\ell+1)^d}{\theta } \overleftarrow{\eta}_x^\ell  \overrightarrow{\eta}_{x+e_i}^\ell \Big\}\Big].
\end{align*}
We conclude the proof by choosing $\gamma = \varepsilon_0 \theta$ for $\varepsilon_0 > 0$ small enough.

Finally, we prove \eqref{h ell 1}. For any $\gamma > 0$, 
\begin{align*}
	&E_{\mu^n_s} \Big[ \frac{1}{ n^2 \theta}  \sum_{y,z \in \T_n^d \atop |y-z|=1} h_{y,z}^\ell (\mathbf{1}, \eta)^2 \Big] \\
	&\leq  \frac{H(\mu^n_s | \nu^n_{\rho_*})}{\gamma} + \frac{1}{\gamma} \log E_{\nu^n_{\rho_*}} \Big[ \exp \Big\{  \frac{\gamma}{ n^2 \theta}  \sum_{y,z \in \T_n^d \atop |y-z|=1} h_{y,z}^\ell (\mathbf{1}, \eta)^2 \Big\}\Big] \\
	&\leq \frac{H(\mu^n_s | \nu^n_{\rho_*})}{\gamma} + \frac{1}{\gamma (2\ell+1)^d} \sum_{y,z \in \T_n^d \atop |y-z|=1} \log E_{\nu^n_{\rho_*}} \Big[ \exp \Big\{  \frac{\gamma (2\ell+1)^d}{n^2 \theta}  h_{y,z}^\ell (\mathbf{1}, \eta)^2 \Big\}\Big].
\end{align*}
We conclude the proof by choosing $\gamma = \varepsilon_0 n^2 \theta/ (\ell^d g_d (\ell))$ for $\varepsilon_0 > 0$ small enough.

\bibliographystyle{plain}
\bibliography{bibliography.bib}

\begin{thebibliography}{10}

\bibitem{bernardin2004fluctuations}
C.~Bernardin.
\newblock Fluctuations in the occupation time of a site in the asymmetric
  simple exclusion process.
\newblock {\em The Annals of Probability}, 32(1B):855--879, 2004.

\bibitem{bernardin2016occupation}
C.~Bernardin, P.~Gon{\c{c}}alves, and S.~Sethuraman.
\newblock Occupation times of long-range exclusion and connections to {KPZ}
  class exponents.
\newblock {\em Probability Theory and Related Fields}, 166(1-2):365--428, 2016.

\bibitem{chang2004occupation}
C.~C. Chang, C.~Landim, and T-Y Lee.
\newblock Occupation time large deviations of two-dimensional symmetric simple
  exclusion process.
\newblock {\em The Annals of Probability}, 32(1B):661--691, 2004.

\bibitem{de1986reaction}
A.~De~Masi, P.~A. Ferrari, and J.~L. Lebowitz.
\newblock Reaction-diffusion equations for interacting particle systems.
\newblock {\em Journal of statistical physics}, 44(3):589--644, 1986.

\bibitem{erhard2024nonequilibrium}
D.~Erhard, T.~Franco, and T.~C. Xu.
\newblock Nonequilibrium joint fluctuations for current and occupation time in
  the symmetric exclusion process.
\newblock {\em Electronic Journal of Probability}, 29:1--53, 2024.

\bibitem{feng2006large}
J.~Feng and T.~G. Kurtz.
\newblock {\em Large deviations for stochastic processes}.
\newblock Number 131. American Mathematical Soc., 2006.

\bibitem{fontes2021additive}
L.~R. Fontes and T.~C. Xu.
\newblock Additive functionals of exclusion processes from non-equilibrium.
\newblock {\em arXiv preprint arXiv:2111.08804}, 2021.

\bibitem{gao2024moderate}
F.~Q. Gao and J.~Quastel.
\newblock Moderate deviations for lattice gases with mixing conditions.
\newblock {\em Electronic Journal of Probability}, 29:1--23, 2024.

\bibitem{gao2025deviation}
F.~Q. Gao and J.~Quastel.
\newblock Deviation inequalities and moderate deviations for the symmetric
  exclusion process.
\newblock In {\em Annales de l'Institut Henri Poincare (B) Probabilites et
  statistiques}, volume~61, pages 2293--2316. Institut Henri Poincar{\'e},
  2025.

\bibitem{gonccalves2013scaling}
P.~Gon{\c{c}}alves and M.~Jara.
\newblock Scaling limits of additive functionals of interacting particle
  systems.
\newblock {\em Communications on Pure and Applied Mathematics}, 66(5):649--677,
  2013.

\bibitem{gonccalves2024clt}
P.~Gon{\c{c}}alves, M.~Jara, R.~Marinho, and O.~Menezes.
\newblock {CLT} for {NESS} of a reaction-diffusion model.
\newblock {\em Probability Theory and Related Fields}, 190(1):337--377, 2024.

\bibitem{jara2018non}
M.~Jara and O.~Menezes.
\newblock Non-equilibrium fluctuations of interacting particle systems.
\newblock {\em arXiv preprint arXiv:1810.09526}, 2018.

\bibitem{jaram18nonequilireaction}
M.~Jara and O.~Menezes.
\newblock Non-equilibrium fluctuations for a reaction-diffusion model via
  relative entropy.
\newblock {\em Markov Processes And Related Fields}, 26(1):95--124, 2020.

\bibitem{kipnis1987fluctuations}
C~Kipnis.
\newblock Fluctuations des temps d'occupation d'un site dans l'exclusion simple
  sym{\'e}trique.
\newblock In {\em Annales de l'IHP Probabilit{\'e}s et statistiques},
  volume~23, pages 21--35, 1987.

\bibitem{klscaling}
C.~Kipnis and C.~Landim.
\newblock {\em Scaling limits of interacting particle systems}, volume 320.
\newblock Springer Science \& Business Media, 2013.

\bibitem{kipnis1986central}
C.~Kipnis and S.~R.~S. Varadhan.
\newblock Central limit theorem for additive functionals of reversible markov
  processes and applications to simple exclusions.
\newblock {\em Communications in Mathematical Physics}, 104(1):1--19, 1986.

\bibitem{landim1992occupation}
C.~Landim.
\newblock Occupation time large deviations for the symmetric simple exclusion
  process.
\newblock {\em The Annals of Probability}, pages 206--231, 1992.

\bibitem{li2008upper}
Z.~Li and M.~Mao.
\newblock Upper bound on the occupation time in the simple exclusion process.
\newblock {\em Theoretical and Mathematical Physics}, 156(1):1089--1100, 2008.

\bibitem{liggettips}
T.~M. Liggett.
\newblock {\em Interacting particle systems}, volume 276.
\newblock Springer Science \& Business Media, 2012.

\bibitem{quastel2002central}
J.~Quastel, H.~Jankowski, and J.~Sheriff.
\newblock Central limit theorem for zero-range processes.
\newblock {\em Methods and applications of analysis}, 9(3):393--406, 2002.

\bibitem{seppalainen2003transience}
T.~Sepp{\"a}l{\"a}inen and S.~Sethuraman.
\newblock Transience of second-class particles and diffusive bounds for
  additive functionals in one-dimensional asymmetric and exclusion processes.
\newblock {\em The Annals of Probability}, 31(1):148--169, 2003.

\bibitem{sethuraman2000central}
S.~Sethuraman.
\newblock Central limit theorems for additive functionals of the simple
  exclusion process.
\newblock {\em Annals of Probability}, pages 277--302, 2000.

\bibitem{sethuraman2006centralcor}
S~Sethuraman.
\newblock Correction: Central limit theorems for additive functionals of the
  simple exclusion process.
\newblock {\em The Annals of Probability}, 34(1):427--428, 2006.

\bibitem{sethuraman2006superdiffusivity}
S.~Sethuraman.
\newblock Superdiffusivity of occupation-time variance in 2-dimensional
  asymmetric exclusion processes with density $\rho= 1/2$.
\newblock {\em Journal of statistical physics}, 123(4):787--802, 2006.

\bibitem{sethuraman1996central}
S.~Sethuraman and L.~Xu.
\newblock A central limit theorem for reversible exclusion and zero-range
  particle systems.
\newblock {\em The Annals of Probability}, 24(4):1842--1870, 1996.

\bibitem{whitt2007proofs}
W.~Whitt.
\newblock Proofs of the martingale {FCLT}.
\newblock {\em Probability Surveys}, 4:268--302, 2007.

\bibitem{xu2025nonequilibrium}
T.~C. Xu and L.~Zhao.
\newblock Nonequilibrium fluctuations for the occupation time of the {SSEP} in
  $d \geq 2$.
\newblock {\em arXiv preprint arXiv:2512.09424}, 2025.

\bibitem{xue2025central}
X.~F. Xue.
\newblock Central limit theorems and moderate deviations for additive
  functionals of {SSEP} on regular trees.
\newblock {\em arXiv preprint arXiv:2504.15581}, 2025.

\bibitem{xue2026equilibrium}
X.~F. Xue.
\newblock Equilibrium moderate deviations for occupation times of {SSEP} on
  regular trees.
\newblock {\em Bernoulli}, 32(1):823--847, 2026.

\bibitem{zhao2025moderate}
L.~J. Zhao.
\newblock Moderate deviation principles for the {WASEP}.
\newblock {\em Journal of Statistical Physics}, 192(6):85, 2025.

\bibitem{zhao2026moderate}
L.~J. Zhao.
\newblock Moderate deviation principles for a reaction diffusion model in
  non-equilibrium.
\newblock {\em Bernoulli}, 32(2):1098--1121, 2026.

\end{thebibliography}
\end{document}